\documentclass[10pt,oneside,a4paper]{article}
\usepackage{titlesec}
\usepackage{authblk}
\titleformat{\subsection}[runin]{\normalfont\bfseries}{\thesubsection.}{.5em}{}[.~ ]
\titlespacing{\subsection}{0pt}{1.5ex plus .1ex minus .2ex}{0pt}
\usepackage{graphicx}
\usepackage{subcaption}
\usepackage{wrapfig}
\usepackage{amsfonts,amssymb,amsmath,amsthm}
\usepackage{enumerate}
\usepackage{url}
\usepackage[mathcal,mathscr]{euscript}
\usepackage[dvipsnames]{xcolor}
\usepackage{hyperref}
\hypersetup{colorlinks=true,linkcolor=Magenta,citecolor=PineGreen}
\usepackage{float}
\usepackage{tikz-cd}
\usepackage[most]{tcolorbox}
\usepackage{epigraph}

\theoremstyle{plain}

\newtheorem{lemma}[equation]{Lemma}

\newtheorem{prop}[equation]{Proposition}
\newtheorem{cor}[equation]{Corollary}

\theoremstyle{definition}
\newtheorem{exe}[equation]{Example}
\newtheorem{defi}[equation]{Definition}

\newtheorem{rmk}[equation]{Remark}
\newcommand{\pp}{{\pmb{p}}}
\newcommand{\qq}{{\pmb{q}}}
\newcommand{\xx}{{\pmb{x}}}

\newcommand{\uu}{\pmb{u}}
\newcommand{\vv}{\pmb{v}}

\newcommand{\ee}{\varepsilon}

\newcommand{\ta}{\mathrm{\mathrm{ta}}}
\newcommand{\tn}{\mathrm{\mathrm{Tn}}}

\newcommand{\SU}{\mathrm{\mathrm{SU}}}
\newcommand{\UU}{\mathrm{\mathrm{U}}}
\newcommand{\PU}{\mathrm{\mathrm{PU}}}

\newcommand{\re}{\mathrm{\mathrm{Re}\,}}

\newcommand{\PP}{\mathbb{P}}
\newcommand{\HH}{\mathbb{H}}
\newcommand{\BB}{\mathbb{B}}
\newcommand{\EE}{\mathbb{E}}
\newcommand{\CC}{\mathbb{C}}

\newcommand{\KK}{\mathbb{K}}
\newcommand{\FF}{\mathbb{F}}

\newcommand{\SP}{\mathbb{S}}
\newcommand{\RR}{\mathbb{R}}

\newcommand{\DD}{\mathbb{D}}
\newcommand{\de}{{\mathrm d}}

\newcommand{\Lin}{\mathrm{Lin}}

\newcommand{\proj}{\mathrm{proj}}

\newcommand{\id}{\mathrm{id}}

\title{\vspace{-15mm}\fontsize{16pt}{10pt}\selectfont\textbf{Geometry over algebras}}

\author{
	\large
	\textsc{Hugo Cattarucci Botós}\thanks{Supported by S\~ao Paulo Research Foundation (FAPESP)}\\
	\normalsize \href{}{hugocbotos@usp.br}\\
	\normalsize{Departamento de Matem\'atica, ICMC, Universidade de S\~ao Paulo, S\~ao Carlos, Brasil}	
	\vspace{-5mm}
}
\date{}

\begin{document}
	
\maketitle

\begin{center}
	\large\textbf{Abstract}
\end{center}
    We study geometric structures arising from Hermitian forms on linear spaces over real algebras beyond the division ones. Our focus is on the dual numbers, the split-complex numbers, and the split-quaternions. The corresponding geometric structures are employed to describe the spaces of oriented geodesics in the hyperbolic plane, the Euclidean plane, and the round $2$-sphere. We also introduce a simple and natural geometric transition between these spaces. Finally, we present a projective model for the hyperbolic bidisc, that is, the Riemannian product of two hyperbolic discs.

    \section{Introduction}
    Following \cite{coordinatefree}, classic geometries emerge from a linear space endowed with a Hermitian form. Typical examples of such geometries are the real/complex/quaternionic projective spaces with Fubini-Study metric, the real/complex/quaternionic hyperbolic spaces, the de Sitter spaces, and anti-de Sitter spaces, among others. Here, we extend the framework of classic geometries to the case of linear structures over real algebras other than the real numbers, the complex numbers, and the quaternions. This is necessary if one wants to, for example, describe natural geometric structures on the spaces of geodesics in usual classic geometries (for the spaces of geodesics in spherical, Euclidean, and hyperbolic geometries, see Section \ref{section space of geodesics}). As in \cite{coordinatefree}, we
    take the coordinate-free route and describe (pseudo)-Riemannian concepts and formulas in a simple algebraic form which is well suited, say, for scientific computation.
    
    The algebras we consider here, besides the associative real division algebras, are the simplest associative unital finite-dimensional involutive real algebras: split-complex numbers $\RR[x]/(x^2-1)$ (also known as hyperbolic numbers), dual numbers $\RR[x]/(x^2)$, and split-quaternions. The reason why we cling to these algebras is that the linear algebra over them is not too ill-behaved (see Section  \ref{section linear algebra}). Moreover, they are enough to describe the above mentioned spaces of geodesics.
    
    Geometries over split-complex and dual numbers were previously studied by S.~Trettel \cite{tre} from a homogeneous spaces approach. In contrast, we work with projective spaces (whose definition is, essentially, the usual one) where geometric structures are induced from a Hermitian form. So, our work is to \cite{tre} as \cite{coordinatefree} is to the usual symmetric/homogeneous approach to projective geometries. For instance, S.~Trettel develops a theory of transition of geometries showing how the one-parameter family of algebras $\RR[x]/(x^2+\delta)$ provides a transition between geometries over complex (${\delta>0}$), dual numbers ($\delta=0$), and split-complex ($\delta<0$) hyperbolic geometries. We describe this transition inside the split-quaternionic projective spaces. It is important to mention that geometries constructed from this one-parameter family of algebras appear in the work of J.~Danciger \cite{dan1}, \cite{dan2}, \cite{dan3}. Among several other results, J.~Danciger presents a transition between the three dimensional hyperbolic and anti-de Sitter spaces passing through a pipe geometry.
    
    Curiously, the hyperbolic bidisc (product of two Poincaré discs) appears as a projective classic geometry in the present context. The bidisc is an important space when it comes to uniformization questions in dimension $4$ \cite{cgr1}. For instance, the known examples of disc bundles uniformized by the bidisc support a bidisc variant of the Gromov-Lawson-Thurston conjecture \cite{glt}. The original GLT conjecture says that an oriented disc bundle over a closed oriented surface of genus~$\ge2$ admits a complete metric of constant negative curvature (a real hyperbolic structure) if, and only if, the Euler number $e$ of the bundle satisfies $|e| \leq |\chi|$, where $\chi$ is the Euler characteristic of the surface. However, all known examples of disc bundles uniformized by the complex hyperbolic space (see~\cite{discbundles}, \cite{bgr}, \cite{gkl}) and by the bidisc (see \cite{cgr1}) support that the GLT conjecture might hold also for these geometries. Describing the bidisc as a projective classic geometry may be a step towards understanding the relationships between such versions of the GLT conjecture; for example, the existence of a transition between the real, the complex and the bidisc hyperbolic geometries inside of a bigger classic geometry might connect the different versions of the GLT conjecture.

	\section{Linear algebra over real algebras}
	\label{section linear algebra}
	The goal of this section is establishing basic linear algebra tools to properly develop projective geometry over some non-division algebras.
	
	\subsection{Finite dimensional real algebras}
	Consider a real finite-dimensional unital associative algebra $\FF$. There is a natural $\RR$-algebra embedding $T:\FF \to \Lin_\RR(\FF,\FF)$ given by $a \mapsto T_a$, where $T_a(x):=ax$. So, left and right zero divisors coincide and a left inverse is also a right inverse. Denote by $\FF_z$ the set of zero-divisors and by $\FF^\times$ the set of units.
	
	\begin{prop} $\FF =\FF_z \sqcup \FF^\times$ 
	\end{prop} 
	\begin{proof} Clearly, $\FF_z \cap \FF^\times = \emptyset$.
		Take $a \in \FF\setminus \FF^\times$. The map $T_a:\FF \to \FF$, $x \mapsto ax$, is $\RR$-linear. It cannot be surjective because that would imply $a\in\FF^\times$. Hence, its kernel is non-trivial. So, $a \in \FF_z$.
	\end{proof}
	
	\begin{prop} The subset $\FF_z$ of $\FF$ is a non-trivial real algebraic set. 
	\end{prop} 
	\begin{proof} Consider the map $p:\FF \to \RR$, $a\mapsto\det(T_a)$.
	Then $a \in \FF_z$ if, and only if, $p(a)=0$. Since $p$ is a (non-zero, several variables) real polynomial, $\FF_z = \{x \in \FF: p(x)=0\}$ is a non-trivial real algebraic set.
	\end{proof}
	
	A real algebraic set can be written as a union of finite smooth manifolds \cite{sh} and we define its dimension as the largest dimension among such manifolds.
	
	\begin{cor} The subset $\FF_z$ of $\FF$ is a finite union of manifolds of dimension smaller than $\dim_\RR \FF$. In particular, $\FF^\times$ is open and dense in $\FF$.
	\end{cor} 

	From now on we assume that $\FF$ has an involutive structure: there is an algebra antiautomorphism $x \mapsto x^\ast$ of $\FF$ such that ${x^\ast}^\ast = x$. By antiautomorphism we mean that this map is an $\RR$-linear isomorphism, $1^*=1$, and $(xy)^*=y^*x^*$. An element $x$ is called self-adjoint when $x^\ast = x$. With a single exception (see Section \ref{section bidisc}), we will restrict ourselves to involutive real algebras where the self-adjoint elements of $\FF$ are exactly the real numbers.
	
	In the case when the quadratic form $N(x):= xx^\ast$ is non-degenerate, these algebras are called (associative) composition algebras. When $N$ is definite, then $\FF$ is one of $\RR,\CC,\HH$, where $\HH$ stands for the quaternions. Otherwise, the algebra is either the split-complex numbers $\CC_s$ or the split-quaternions $\HH_s$:
	
	\medskip
	
	$\bullet$ {\bf Split-complex numbers} $\CC_s:=\RR + j \RR$, with $j^2=1$ and involution $(x+j y)^*= x-jy$;
	
	\smallskip
	
	$\bullet$ {\bf Split-quaternions} $\HH_s:=\RR + i \RR+ j \RR + k \RR$, with $$i^2=-1,\quad j^2=1,\quad k^2 =1,$$ $$ ij = k,\quad ij=-ji,\quad ik=-ki,\quad jk= - kj,$$
		and involution $(x+iy+zj+wk)^* = x-iy-zj-wk$.
	
	\medskip
	
	There are also cases where $N$ is degenerate. For example, we have the
	
	\medskip
	
	$\bullet$ {\bf Dual numbers} $\DD:=\RR + \ee \RR$, with $\ee^2 =0$ and involution $(x+\ee y)^* = x-\ee y$.
	
	\medskip
	
	Note that the split-quaternions contain copies of the complex, split-complex, and dual numbers (the last one happens, for example, taking $\ee : = i+j$).
	
	The split-complex numbers $\CC_s$ can be naturally identified with the algebra $\RR \times \RR$ endowed with the involution $(a,b)^\times = (b,a)$. The isomorphism is given by the map \begin{align*}
	\RR + j\RR &\to \RR \times \RR\\
	x+jy &\mapsto (x+y,x-y)
	\end{align*}
	From this identification, we obtain that the set units $\CC_s^\times$ of the split-complex numbers is $\RR^\times \times \RR^\times$. The units of the dual numbers $\DD$ are of the form $a+\ee b$, where $a \in \RR^\times$.
	
	\subsection{Hermitian form}
	
	All modules in this paper are finite-dimensional free left-modules. Observe that, when $V$ is also free, the concept of dimension of $V$ as an $\FF$-module is well defined. Indeed, let $V$ be a free $\FF$-module with basis $e_1,\ldots,e_n$. Since each $\FF e_i$ is a real vector space with dimension $\dim_\RR \FF$, the number $n = \dim_\RR V/\dim_\RR \FF$ does not depend on the choice of basis. 
	
	\begin{defi}
		A Hermitian form on $V$ is a map $\langle \cdot,\cdot \rangle: V \times V \to \FF$ satisfying the following properties:
		\begin{itemize}
			\item $\langle u+v,w \rangle = \langle u,w\rangle +\langle v,w\rangle, \quad 
			u,v,w\in V$;
			\item $\langle zu,w \rangle = z \langle u,v \rangle, \quad u,v \in V, z\in\FF$;
			\item  $\langle u,v \rangle^\ast =\langle v,u \rangle, \quad u,v \in V$.
		\end{itemize}
	\end{defi}

	In particular $\langle u,u \rangle \in \RR$ for all $u \in V$.
	
	\medskip
	
	From now on, $V$ is a finite-dimensional left $\FF$-module equipped with a Hermitian form. An orthonormal basis consists of $b_1,\ldots,b_n\in V$ such that $\langle b_i,b_i \rangle=\pm 1$, $\langle b_i, b_j \rangle=0$ for $i \neq j$, and $V = \FF b_1 \oplus \cdots \oplus \FF b_n$.
	
	\begin{defi} A Hermitian form is non-degenerate if zero is the only vector perpendicular to all vectors.
	\end{defi}

	\begin{lemma} \label{lemma orthogonal vectors extension} Consider a finite-dimensional free $\FF$-module $V$ equiped with a non-degenerate Hermitian form. 
	If $W$ is a proper subspace of $V$ that admits a orthonormal basis, then there exists $u \in V$ such that $\langle u,u \rangle \ne0$ and $\langle W,u \rangle =0$.
	{\rm(}In particular, there always exists $u \in V$ such $ \langle u, u \rangle \neq 0$.{\rm)}
	\end{lemma}
	\begin{proof}We assume $\FF\neq \RR$ (otherwise, the fact is trivial).
	
	Fix an orthonormal basis $b_1,\ldots, b_m$ for $W$.
	
	Let
	$W^\perp := \{u \in V\mid \langle u, W \rangle =0\}.$
	Note that $W \cap W^\perp = 0$ and that for  each $u \in V$, the vector $$u' := u - \sum_i \frac{\langle u,b_i \rangle}{\langle b_i,b_i \rangle} b_i$$ belongs to  $W^\perp$. Therefore, $V=W\oplus W^\perp$.

	Suppose that for all $v \in W^\perp$ we have $\langle v,v \rangle = 0$. Let us show that such assumption leads to a contradiction, thus proving the result.
		
	Fix $u \in W^\perp$. Note that $\langle u+h,u+h\rangle=0$ for all $h\in W^\perp$. So, $\langle u,h \rangle + \langle h,u \rangle =0$ for all $h \in W^\perp$. Clearly, $\langle u,h \rangle + \langle h,u \rangle =0$ for all $h \in V$.
		
	If $\FF$ is the split-complex or the complex numbers, there is $j \in \FF^\times$ such that $j^*=-j$. So, $\langle u,jh \rangle + \langle jh,u \rangle =0$ which implies
	$\langle u,h \rangle - \langle h,u \rangle =0$ for all $h\in V$, that is, $\langle u,h\rangle=0$ for all $h\in V$. Therefore, $u=0$ (the Hermitian form is non-degenerate), contradicting $W^\perp \neq 0$.
		
	For the quaternions and split-quaternions, we have the numbers $i,j,k \in \FF^\times$ that anti-commute among themselves and satisfy $$i^*=-i,\quad j^*=-j, \quad k^*=-k.$$ 
	For each of this numbers, we obtain
	$$i\langle u,h \rangle =\langle u,h \rangle i,$$	
	$$j\langle u,h \rangle =\langle u,h \rangle j,$$
	$$k\langle u,h \rangle =\langle u,h \rangle k.$$
	Since only real numbers commute with $i,j,k$, we conclude that $\langle u,h \rangle=0$ for all $h$. Hence $u=0$, contradicting $W^\perp \neq 0$.
		
	For the dual numbers, we have $\FF = \RR \oplus \RR\epsilon$, $\epsilon^2=0$. Proceeding as above, we obtain the identity $\epsilon\langle u,h \rangle - \epsilon \langle h,u \rangle =0$ for all $h \in V$ which implies $\langle\epsilon u,h \rangle =0$ and, therefore, $\epsilon u=0$. Since $V$ is free, this implies $u\in\epsilon V$. Thus, $W^\perp$ is a subspace of $\epsilon V$ implying
	$\epsilon W\oplus W^\perp \subset \epsilon V$. In particular, 
	$$\dim_\RR \epsilon W + \dim_\RR W^\perp  \leq  \dim_\RR \epsilon V.$$
	Since $W$ is free, $2\dim_\RR \epsilon W = \dim_\RR W$. Thus, it follows from $V = W \oplus W^\perp$ that
	$$\frac{\dim_\RR  V - \dim_\RR W^\perp}2 + \dim_\RR W^\perp \leq \dim_\RR \epsilon V$$
	and, therefore,
	$$ \dim_\RR W^\perp \leq 2 \dim_\RR \epsilon V - \dim_\RR V.$$
	We reached a contradiction: $2 \dim_\RR \epsilon V - \dim_\RR V = 0$ because $V$ is free, but $\dim_\RR W^\perp >0$.
		
	Finally, we conclude that there exists $u \in W^\perp$ such that $\langle u,u\rangle \neq 0$.
	\end{proof}
    
    We have the following corollaries:
	
	\begin{cor} \label{exists orthogonal basis} If the Hermitian form is non-degenerate then the finite-dimensional free $\FF$-module $V$ has an orthonormal basis.
	\end{cor}
	
	\begin{cor} \label{exists orthogonal basis part 2} Let $V$ be a finite dimensional free $\FF$-module endowed with a non-degenerate Hermitian form. If $W$ is a free submodule of $V$ where the Hermitian form is non-degenerate, then any orthonormal basis of $W$ can be completed to an orthogonal basis of $V$.
	\end{cor}

	\begin{cor} \label{canonical hermitian forms are non-degenerate} If $V$ has an orthonormal basis then the Hermitian form is non-degenerate.
	\end{cor}

	\subsection{Good points}
	
	\begin{defi} We say that $u \in V$ is a {\bf good point} if there exists a basis for $V$ such that $\sum_i u_i \FF = \FF$, where $u_1,\ldots,u_n$ are the coordinates of $u$ on such basis. We denote the set of all good points by~$V^\bullet$. 
	\end{defi}
	
	\noindent
	(Clearly, for the concept of good point to be well defined, we need $V$ to be free.) Alternatively, a point $u$ is good if for every non-degenerate Hermitian form $\langle \cdot,\cdot \rangle$ on $V$ there exists $v\in V$ such that $\langle u,v \rangle =1$ (see Proposition \ref{goodpoint alternative definition}).
	
	\begin{prop} If $u\in V$ is a good point, then for every basis of $V$ we have $\sum_i u_i \FF = \FF$, where $u_1,\ldots,u_n$ are the coordinates of $u$ on such basis.
	\end{prop}	
	\begin{proof} Take a basis $e_i$ such that $$u= \sum_i u_i e_i \quad \text{and} \quad \sum_i u_i \FF = \FF.$$ 
		The second condition means that there are $v_1,\ldots,v_n \in \FF$ such that $\sum_i u_iv_i = 1$.
		
		Consider another basis $f_j$. We have $f_j = \sum_i \alpha_{ij} e_i$ and $e_j = \sum_i \beta_{ij} f_i$. Therefore, for each $i,j$, we have
		$$\sum_k \alpha_{ki} \beta_{jk} = \sum_k \beta_{ki} \alpha_{jk}  =\delta_{ij}.$$
		The coordinates of $u$ on the basis $f_k$ are given by $\widetilde {u_k} = \sum_i u_i \beta_{ki}$. For $ \widetilde {v_k}: = \sum_j \alpha_{jk} v_j$ we obtain
		$$\sum_k\widetilde{u_k} \widetilde{v_k} = \sum_{i,j,k} u_i \beta_{ki}\alpha_{jk} v_{j}=\sum_{i,j} \delta_{ij} u_i v_j =1,$$
		thus proving that 
		$$\sum_i \tilde{u_i}\FF = \FF.$$
	\end{proof}
	
	\begin{prop} The set of good points $V^\bullet$ is an open dense subset of $V$.
	\end{prop}
	\begin{proof} We may suppose $V\!\!:=\!\FF^n$ because $V$ is free. For  $u \in V^\bullet$ there is $v \in \FF^n$ such that $\sum_i u_iv_i =1$. Note that
		$$U=\big\lbrace x \in V: \sum_i x_iv_i \in \FF^\times \big\rbrace \subset V^\bullet$$
		is an open neighborhood of $u$, since $\FF^\times$ is open in $\FF$. Thus, $V^\bullet$ is open. To see that it is also dense, just note that $V^ \bullet$ contains the dense subset $(\FF^\times)^n$, where we are using that $\FF^\times$ is dense in $\FF$.
	\end{proof}
	
	\begin{prop}\label{goodpoint alternative definition} Assume that $V$ is equipped with a non-degenerate Hermitian form. A point $u\in V$ is good if, and only if, the map $h \mapsto\langle u,h \rangle$ from $V$ to $\FF$ is surjective.
	\end{prop}
	\begin{proof} Consider an orthogonal basis $b_1,\ldots,b_n$ and define $c_i = \langle b_i,b_i \rangle$. 
		Take $u \in V^\bullet$. Writing $u=\sum_i u_i b_i$, let $\alpha_i \in \FF$ be such that
		$$\sum_i u_i\alpha_i =1.$$
		The vector
		$$h = \sum_i  c_i^{-1}\alpha_i^\ast b_i$$
		satisfies
		$$\langle u,h \rangle = 1.$$
		Thus, the map $\langle u,- \rangle$ is surjective. The converse is analogous.
	\end{proof}
	
	\subsection{Orthogonal complement of a good point}
	
	\begin{prop} \label{good point gives a basis} Let $p \in V^\bullet$. There exist $b_2,\ldots,b_n \in V$ such that $p,b_2,\ldots,b_n$ is a basis.
	\end{prop}
	\begin{proof} We assume $V= \FF^n$. Consider the Hermitian form 
		$$\langle x,y \rangle = \sum_i x_iy_i^*.$$
		
		If $\langle p,p \rangle  \neq 0$, then the result follows from Lemma \ref{lemma orthogonal vectors extension}. Thus, we may assume $\langle p,p \rangle = 0$.
		
		Since $p \in (\FF^n)^\bullet$, there exists $q \in \FF^n$ such that $\langle p,q \rangle =1$. 
		
		Consider the  $\FF$-linear space $W = \FF p + \FF q$. 
		Note that $\FF p \cap \FF q=0$. Indeed, if $xp+yq = 0$ then $y=0$ since $x\langle p,p \rangle + y\langle q,p \rangle = y$. We obtain $xp=0$ which implies $\langle xp,q\rangle=x=0$. Thus, $W = \FF p \oplus \FF q$. We will show that $W$ admits an orthonormal basis and the result will follow from Lemma \ref{lemma orthogonal vectors extension}.
		
		First we assume $\langle q,q \rangle \neq 0$. We may suppose $\langle q,q \rangle =\pm 1$. For $p' := p-\frac{\langle p,q \rangle}{\langle q,q \rangle}q$, we obtain $W=\FF p' + \FF q$, $\langle p',p' \rangle = \mp 1$, $\langle p',q\rangle =0$. Thus, $p',q$ form an orthonormal basis of $W$.
		
		Finally, we take $\langle q,q \rangle=0$. Take $p'=p+q$ and $q'=p-q$. We have $W= \FF p' \oplus \FF q'$, $\langle p',p' \rangle = -\langle q',q' \rangle =2$ and $\langle p',q' \rangle = 0$.
	\end{proof}

	\begin{prop} \label{existence of p perp} Assume that $V$ is equipped with a Hermitian form. Let $p \in V$ be such that $\langle p,p \rangle \neq 0$. Then $V=\FF p \oplus p^\perp$ and $p^\perp$ is free, where $$p^\perp:=\{v \in V: \langle v, p \rangle =0\}$$
	stands for the orthogonal complement of $p$.
	
	\end{prop}
	\begin{proof} It is easy to see that $V=\FF p \oplus p^\perp$. Indeed, $\FF p \cap p^\perp = 0$ and for every $v \in V$ we have $$v=\frac{\langle v, p\rangle}{\langle p,p \rangle } p + \left(v-\frac{\langle v, p\rangle}{\langle p,p \rangle } p\right).$$
	By Proposition \ref{good point gives a basis} there exist $b_2,\ldots,b_n \in V$ such that $p,b_2,\ldots,b_n$ is a basis of $V$. The vectors 
	$$b_k-\frac{\langle b_k, p\rangle}{\langle p,p \rangle } p$$
	form a basis for $p^\perp$.	
	\end{proof}

	\section{Projective geometry}
	
	In this section we study the smooth structure of projective spaces over algebras. Their tangent spaces and vector fields are easily described via linear tools, analogous to what is done in \cite{coordinatefree}.
	
	Let $V$ be a finite dimensional free $\FF$-module.
	The projective space $\PP_\FF(V)$ is defined as
	$$\PP_{\FF}(V) := V^\bullet/\FF^\times.$$
	
	Whenever there is no possible confusion we omit the subscript $\FF$ and just write $\PP(V)$; given $p \in V^\bullet$, we denote by $\pp\in\PP(V)$ the corresponding point on the projective space.
		
	\begin{prop} The space $\PP(V)$ is a smooth manifold of dimension $\dim_\RR V - \dim_\RR \FF$. Furthermore, the quotient map $V^\bullet \to \PP(V)$ is a principal $\FF^\times$-bundle.
	\end{prop}

	\begin{proof} Note that $\FF^\times$ is a Lie group and $V^\bullet$ is an open subset of $V$.
	Consider the smooth injective map $\lambda:\FF^\times \times V^\bullet \to V^\bullet \times V^\bullet$ defined by $\lambda(\alpha,v) := (\alpha v,v)$. Let us prove that such map is proper. 
	
	Consider a compact $K \subset V^\bullet\times V^\bullet$ and a sequence $(\alpha_n,v_n) \in \lambda^{-1}(K)$. Since $K$ is compact we may assume, without loss of generality, that $(\alpha_n v_n, v_n)$ converges in $K$ to a limit $(w, v)$.
	
	Since $V$ is free, it admits a non-degenerate Hermitian form $\langle -,- \rangle$. Take $h,h' \in V$ such that $\langle v,h \rangle=\langle w,h' \rangle =1$. Consider the scalars $\alpha := \langle w, h \rangle$ and $\beta := \langle v,h' \rangle$.
	
	Note that $$\alpha_n = \frac{\langle \alpha_n v_n,h \rangle}{\langle v_n,h \rangle}  \to \alpha \quad \text{and} \quad  \alpha \beta = \lim \alpha_n \langle v_n,h' \rangle=\langle w,h' \rangle =1.$$
	Thus, every sequence on $\lambda^{-1}(K)$ admits a convergent subsequence and, therefore, the map $\lambda$ is proper.
	
	Since the action of $\FF^\times$ on $V^\bullet$ is free and proper, $\PP(V):=V^\bullet/\FF^\times$ is a smooth manifold and the quotient map $V^\bullet \to \PP(V)$ is a principal $\FF^\times$-bundle.
	\end{proof}
	
	\begin{cor} We have the natural isomorphism $$C^\infty(\PP(V)) \simeq \{f \in C^\infty(V^\bullet): f \text{ is }\FF^\times  \text{-invariant}\}.$$
    \end{cor}
	Based on this corollary we will always think of smooth functions on the projective space as $\FF^\times$-invariant smooth functions on $V^\bullet$. 
	
	\begin{exe}\label{example real, complex and quaternionic projective spaces} If $\FF$ stands for $\RR$, $\CC$ or $\HH$, then we obtain the usual real, complex and quaternionic projective spaces $\PP_{\FF}^n := (\FF^{n+1})^\bullet/\FF^\times$. Observe that the projective lines are spheres in this case: $\PP_\RR^1 \simeq \SP^1$, $\PP_\CC^1 \simeq \SP^2$, $\PP_\HH^1 \simeq \SP^4$.  
	\end{exe}
	\begin{exe}\label{example split-complex projective spaces} If $\CC_s = \RR \times \RR$ stands for the split-complex numbers, then the projective space $\PP_{\CC_s}^n$ is diffeomorphic to $\PP_\RR^n \times \PP_\RR^n$. Indeed, the diffeomophism is  
	\begin{align*}
	\PP_{\CC_s}^n &\to \PP_\RR^n \times \PP_\RR^n\\
	[(x_0,y_0): \cdots:(x_n,y_n)] &\mapsto ([x_0:\cdots:x_n], [y_0:\cdots:y_n])
	\end{align*}
	In particular, the corresponding projective line is the torus $\PP_\RR^1 \times \PP_\RR^1$.
	
	The split-complex projective spaces model the point-hyperplane geometry. Indeed, consider a real vector space $W$ and let $W^\ast$ be its dual space. We define the $\CC_s$-module
	$$V:=(1,0)W \oplus (0,1)W^\ast.$$
	The region $V^\bullet$ of good points is $(1,0)(W\setminus 0) \oplus (0,1)(W^\ast \setminus 0)$ and its projectivization give us $\PP_{\CC_s}(V)=\PP(W) \times \PP(W^\ast)$. Note that a point $(1,0)p+(0,1)\phi$ in this space represents a point and a hyperplane on $\PP_\RR(W)$.  
	\end{exe}
	
	\begin{exe}\label{example dual number projective spaces} If we consider the dual numbers $\DD:= \RR + \ee \RR$, then $\PP_\DD^n$ is the tangent bundle of $\PP_\RR^n$. Indeed, consider a real vector space $W$ and the $\DD$-module $V:=W\oplus \ee W$. As shown in Proposition~\ref{tangent space without form}, the tangent space $T_{\pp} \PP_\RR(W)$ of $\PP_\RR(W)$ at the point $\pp$ can be identified with $\Lin_\RR(\RR p,W/{\RR p})$. Thus we obtain the diffeomorphism 
	\begin{align*} \PP_\DD(V) &\to T \PP_\RR(W)\\
				p+\ee v &\mapsto\varphi_{p+\ee v} 
	\end{align*}
	where $\varphi_{p+\ee v}: \RR p \to V/\RR p$ is the tangent vector at $\pp$ defined by $r p \mapsto r v+\RR p$. Observe that the map is well-defined: if $\zeta:=\alpha+\ee\beta\in\DD^\times$ then $\alpha\ne0$ and $\varphi_{\zeta(p+\ee v)}=\varphi_{\alpha p+(\alpha v+\beta p)\ee}=\varphi_{\alpha p+\alpha v\ee}$.
	
	This map can be better visualized if we endow $W$ with an inner product $\langle \cdot,\cdot \rangle$. In this case, $T_{\pp} \PP_\RR(W) = \Lin_\RR(\RR p, p^\perp)$ and we have the diffeomorphism
	\begin{align*} \PP_\FF(V) &\to T \PP_\RR(W)\\
		p+\ee v &\mapsto \frac{\langle-, p \rangle}{\langle p , p \rangle}\left(v - \frac{\langle v,p\rangle}{\langle p , p \rangle}p\right) 
	\end{align*}

	The dual number projective line is the tangent bundle of a circle, i.e., a cylinder.
	\end{exe}

	\begin{wrapfigure}[20]{r}{6. cm}
		\centering
		\vspace*{-.4cm}
		\includegraphics[scale=.3 ]{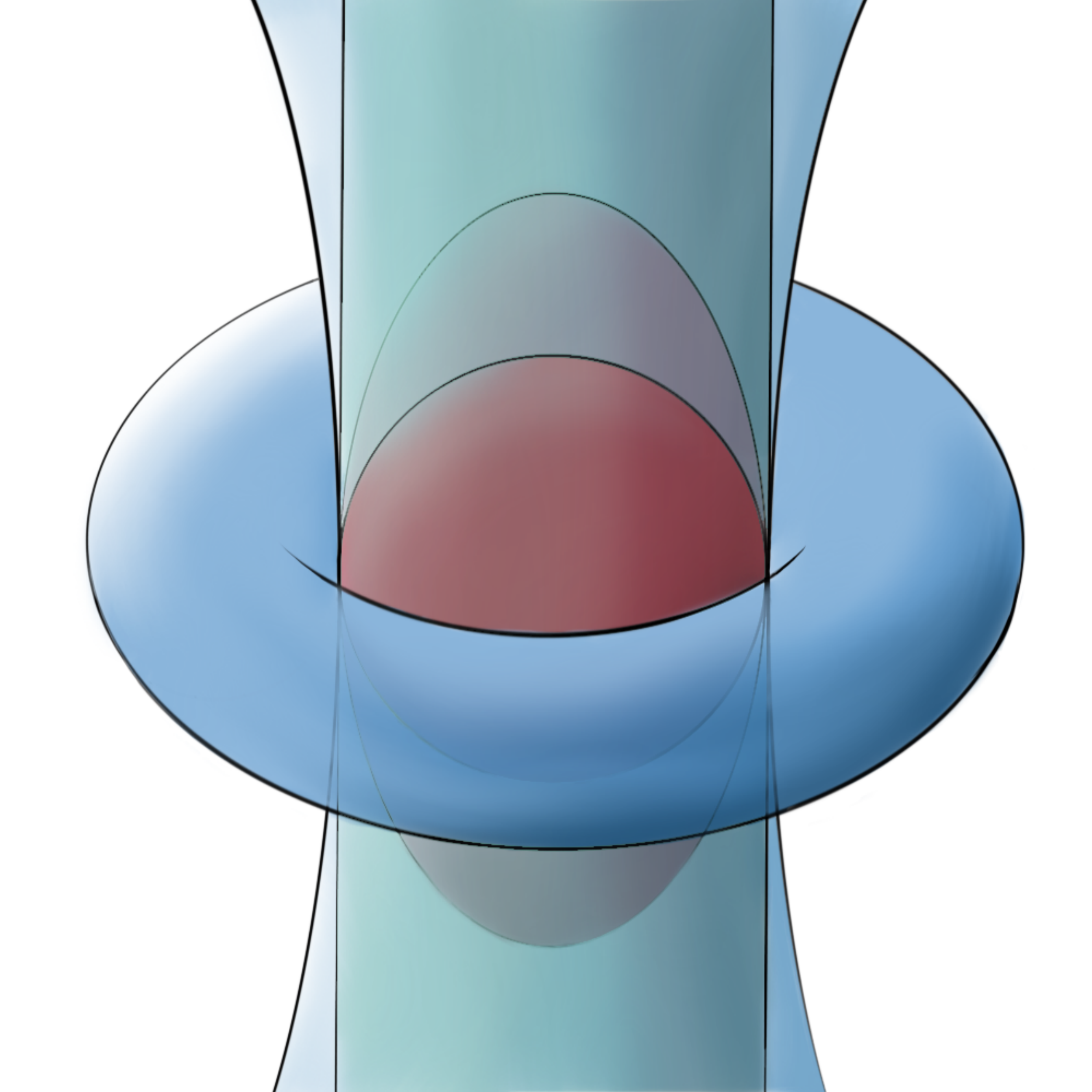}
		\caption{Transition between the split-complex (blue torus), dual number (green cylinder) and complex (red sphere) projective lines}
		\label{transtion}
	\end{wrapfigure}
	
	In the same way that $\RR$ and $\CC$ projective spaces can be embedded in quaternionic projective spaces ($\RR$ and $\CC$ are subalgebras of $\HH$), the projective spaces over $\RR,\CC,\DD,\CC_s$ can be embedded in the split-quaternionic projective space:
	 
	\begin{exe} \label{example split-quaternionic projective spaces}
	The projective space $\PP_{\HH_s}^n$ over the split-quaternions $\HH_s$ is an ambient space for the previously described geometries. Indeed, taking $t\in [0,1]$ and defining $\sigma(t) := (1-t)i+ t j \in \HH_s$, we obtain the one parameter family of subalgebras $\KK_t := \RR + \sigma(t) \RR$ of the split-quaternions. Note that 
		
	\[   
	\KK_t \simeq 
	\begin{cases}
		\RR+i \RR  &\quad\text{for} \quad 0 \leq t<1/2\\
		\RR + \ee \RR  &\quad\text{for} \quad t =1/2\\
		\RR + j \RR  &\quad\text{for} \quad 1/2 < t \leq 1\\ 
	\end{cases}
	\]
	because $\sigma(t)^2 = -(1-t)^2+t^2$. Observe that these algebras are $\CC$, $\DD$ and $\CC_s$, respectively. Thus, we have the following one parameter family of embeddings
	\begin{align*}
	\PP_{\KK_t}^n&\hookrightarrow \PP_{\HH_s}^n\\
	[z_0:\cdots:z_n] &\mapsto [z_0:\cdots:z_n]
	\end{align*}
	Therefore, there is a natural transition between the $\CC$, $\DD$ and $\CC_s$ projective geometries (see Figure \ref{transtion}).
	
	This transition of geometries is described in a different fashion in \cite{tre}.
	\end{exe}
	
	\begin{defi} Given finite dimension free $\FF$-modules $V_1$ and $V_2$, we define $\Lin_\FF(V_1,V_2)$ as the space of all real linear transformations $\phi:V_1 \to V_2$ satisfying $\phi(\alpha v) = \alpha \phi(v)$ for every $\alpha \in \FF$, $v \in V_1$. This space is $\RR$-linear in general and $\FF$-linear when $\FF$ is commutative.  
	\end{defi}
	
    The quotient $V/\FF p$ is free by Proposition \ref{good point gives a basis} and its dimension with respect to $\FF$ is $\dim_\FF V -1$. An element $\phi \in \Lin_\FF(\FF p, V/\FF p)$ is uniquely determined by $\phi(p)$, and thus the real dimension of $\Lin_\FF(\FF p, V/\FF p)$ is $\dim_\RR V - \dim_\RR \FF$.
	
	Given a map $\phi \in \Lin_{\FF}(\FF p, V/\FF p)$ we define a tangent vector $t_\phi \in T_{\pp} \PP(V)$ by the formula:
	$$t_\phi(f) = \frac{d}{d\varepsilon}\Big\vert_{\varepsilon =0} f(p+\varepsilon v),$$
	where $p$ is a representative of $\pp$ and $v$ is a representative of $\phi(p)$.
	Note that the above definition does not depend on the choice of  $v$: if $v,v' \in V $ satisfy $[v]= [v'] = \phi(p)$, then $v-v' = \alpha p$ for some $\alpha \in \FF$ and
	
	$$\frac{d}{d\varepsilon}\Big\vert_{\varepsilon =0} f(p+\varepsilon v) = \frac{d}{d\varepsilon}\Big\vert_{\varepsilon =0} f(p+\varepsilon v')+ \frac{d}{d\varepsilon}\Big\vert_{\varepsilon =0} f(p+\varepsilon \alpha p).$$
	Since for suficiently small $\varepsilon \neq 0$ the element $(1+\varepsilon \alpha)$ is a unit, we conclude that
	$f(p+\varepsilon \alpha p) = f(p)$. Therefore,
	$$\frac{d}{d\varepsilon}\Big\vert_{\varepsilon =0} f(p+\varepsilon v) = \frac{d}{d\varepsilon}\Big\vert_{\varepsilon =0} f(p+\varepsilon v').$$
	Finally, the definition does not depend on the choice of a representative of $\pp$. 
	So, $t_\phi \in T_{\pp} \PP(V)$.
	 
	\begin{prop}\label{tangent space without form} Let $\pp \in \PP(V)$. The map $t:\Lin_\FF(\FF p, V/\FF p) \to  T_{\pp} \PP(V)$ mapping $\phi$ to $t_\phi$ is an $\RR$-isomorphism.
	\end{prop}
	\begin{proof} Note that both spaces have the same real dimension $\dim_\RR V - \dim_\RR \FF$. We just have to prove that $t$ is surjective. Let $\gamma: \RR \to \PP_\FF(V)$ be a smooth curve, $\gamma(0)=\pp$. We lift $\gamma$ around $0$ to a map $\tilde \gamma$ with codomain $V^\bullet$ such that $\tilde \gamma(0)=p$. Hence,
	$$\gamma'(0)f = \frac{d}{d\varepsilon}\Big\vert_{\varepsilon =0} f(\tilde \gamma(\varepsilon)).$$
	Expanding in Taylor series, we have $\tilde\gamma(\varepsilon) = p + \varepsilon\tilde \gamma'(0) + o(\varepsilon)$ and, consequently,
	$$\gamma'(0)f= \frac{d}{d\varepsilon}\Big\vert_{\varepsilon =0} f(p + \varepsilon\tilde \gamma'(0)).$$
	Thus, defining $\phi: \FF p \to V/\FF p$ by the formula $\phi(\alpha p) = [\alpha \tilde \gamma'(0)]$ we conclude that
	$t_\phi = \gamma'(0)$.
	\end{proof}
	With the above proposition in mind, every time we write $T_\pp \PP(V)$ we mean $\Lin_\FF (\FF p, V/\FF p)$.

    \smallskip

	Now consider a Hermitian form $\langle \cdot,\cdot \rangle$ on $V$. The projective space $\PP(V)$ has two distinguished regions
	$$S(V):=\{ \pp \in \PP(V): \langle p,p \rangle =0\},$$
	$$R(V):=\{ \pp \in \PP(V): \langle p,p \rangle  \neq 0\}.$$
	The points of $S(V)$ are called {\bf singular} and, those of $R(V)$, {\bf regular}.

	\begin{prop} If $\pp \in R(V)$, then $$T_\pp \PP(V) \simeq \Lin_\FF(\FF p, p^\perp).$$
	\end{prop}
	\begin{proof} Follows directly from Proposition \ref{existence of p perp}.
    \end{proof}
	Whenever working with a tangent space at a regular point $\pp$ we will think of $T_\pp \PP(V)$ as $\Lin_\FF(\FF p, p^\perp)$.
	\begin{defi}\label{hermitian metric}
	The Hermitian form $\langle \cdot,\cdot \rangle$ induces a Hermitian metric on $R(V)$, defined by 
	$$\langle \phi, \psi  \rangle_\pp = \pm \frac{\langle \phi(p),\psi(p) \rangle}{\langle p,p \rangle}$$
	for $\phi,\psi \in T_\pp \PP(V)$. The sign is to be fixed conveniently. Associated to this Hermitian metric is the pseudo-Riemannian metric
	$$g_\pp(\phi, \psi) = \re \langle \phi,\psi \rangle_\pp,$$
	where $\re u:= (u+u^\ast)/2$.
	\end{defi}
	
	\begin{exe}\label{ex metrics} Let $V:=\FF^{n+1}$ be endowed with the Hermitian form
	$$\langle u,v \rangle := \sum_i u_i v_i^\ast.$$
	Consider the pseudo-Riemannian metric $g$ on the regular region $R(V)$ obtained from the Hermitian metric in Definition \ref{hermitian metric} with positive sign.
	
	For $\FF = \RR,\CC,\HH$, the metric $g$ is Riemannian, the usual Fubini-Study metric. 
	
	For split-algebras $\FF = \CC_s,\HH_s$, the metric $g$ is split, i.e., its signature has the same number of pluses and minuses.
	
	The regular region over dual numbers $\DD$ is the whole projective space, and the signature has $n$ pluses and $n$ zeros. The vectors $\phi$ parallel to the fibers of $\PP_\DD^n \to \PP_\RR^n$, $[u+\ee v] \mapsto [u]$, are the ones with null norm, i.e., $g(\phi,\phi)=0$.
	
	For projective lines, the metrics have the following signatures:

	\begin{center}
	\begin{tabular}{ l l }
		$\PP_\RR^1$ has signature $+$ & $\PP_{\DD}^1$ has signature $+0$  \\ 
		$\PP_\CC^1$ has signature $++$ & $\PP_{\CC_s}^1$ has signature $+-$  \\  
		$\PP_\HH^1$ has signature $++++$ & $\PP_{\HH_s}^1$ has signature $++--$     
	\end{tabular}
	\end{center}

	\end{exe}
	
	\begin{exe} The split-complex projective space (point-hyperplane geometry) arising from $V := (1,0)W \oplus (0,1)W$, as described in the Example \ref{example split-complex projective spaces}, has a natural geometry.
	
	Given two vectors $v_1:=(1,0)w_1 + (0,1) \varphi_1$ and $v_2=(1,0)w_2 + (0,1) \varphi_2$, we have the natural Hermitian form 
	$$\langle v_1,v_2 \rangle := (1,0)\phi_1(v_2)+(0,1)\phi_2(v_1).$$
	In particular, if $v:=(1,0)w + (0,1) \varphi$, then $\langle v,v \rangle = (1,1) \phi(v)$. Thus, the regular region describes the pair of points $[w]$ and hyperplanes $\varphi(x) =0$ such that the point is not in the hyperplane. Taking $W=\RR^{n+1}$ and identifying $W^\ast = W$ via the standard Euclidean metric, the metric associated to the above Hermitian form coincides with the one described in the Example \ref{ex metrics}. That is, the signature of $\PP_{\CC_s}(V)$ is split.
	\end{exe}
	
	\begin{exe}\label{example hyperbolic space} The $n$-dimensional real hyperbolic space $\HH_\RR^n$ is the ball $\{\pp \in \PP_\RR^n: \langle p,p \rangle < 0\}$, where $\langle \cdot, \cdot \rangle$ is the canonical real Hermitian form on $\RR^{n+1}$ with signature $-+ \cdots+$. The hyperbolic Hermitian metric is obtained from Definition \ref{hermitian metric} using the minus sign. The complex and quaternionic hyperbolic spaces are defined likewise.
	
	The region $\mathrm{dS}^n=\{\pp \in \PP_\RR^n: \langle p,p \rangle > 0\}$ with the metric defined above is the projectivization of the de Sitter space, and it is a Lorentz manifold. 
	\end{exe}
	
	Now, let us discuss vector fields. For a regular point $\pp$ we can think of $T_\pp \PP(V)$ as a subset of $\Lin(V, V)$, because $V=\FF p \oplus p^\perp$. More precisely, $T_\pp \PP(V)$ can be seen as the linear maps $\phi\in \Lin(V, V)$ such that $\phi(p) \in p^\perp$ and $\phi(p^\perp)=0$.

	\begin{defi}A vector field on an open subset $U$  of $R(V)$ is a smooth map $X: U \to \Lin(V,V)$ satisfying $X(\pp) \in T_\pp V$ for all $\pp \in U$. We denote the space of all vector fields by $\mathfrak X(U)$.
	\end{defi}

	Among the vector fields there are special ones called called {\bf spread vector fields}.
	Given $\pp \in R(V)$ we define the two projections
	$\pi'[\pp]: V \to \FF p$ and $\pi[\pp]: V \to p^\perp$ by
	$$\pi'[\pp]v = \frac{\langle v, p \rangle}{\langle p,p \rangle}p \quad \text{and} \quad \pi[\pp]v = v - \pi'[\pp]v.$$
    Both formulas are well defined because they do not depend on the choice of a representative $p$ of~$\pp$. 
	
	\begin{defi}\label{spread vector fields} A spread vector field $T$ is a vector field
	defined by 
		$$T_\qq := \pi[\qq]\circ t \circ \pi'[\qq]$$
	for a given $t\in \Lin(V,V)$. The vector field $T$ is said to be spread from $t$.
 	\end{defi}

	The importance of spread vector fields lies on the fact that if we have $\phi \in T_\pp \PP(V)$ for a regular~$\pp$, then the spread $\Phi$ from $\phi$ is a vector field satisfying $\phi = \Phi_\pp$ (in other words, we have a natural way to extend vectors to vector fields). Furthermore, calculating tensors is largely simplified by the use of spread vector fields; this is analogous to what happens in Lie groups when working with left-invariant vector fields.
	\section{Connection and geodesics}
	
	Following \cite{coordinatefree}, we give an algebraic description of the (pseudo-)Riemannian geometry on the previously discussed projective spaces.
	
	\subsection{Levi-Civita connection}
	 A vector field $X$ on the open set $U \subset R(V)$ is, in particular, a smooth map $X: U \to \Lin(V,V)$. We remind that the quotient map $\proj: V^\bullet \to \PP(V)$ defines a principal $\FF^\times$-bundle. For $\tilde U := \proj^{-1}U$ there is a smooth map $\tilde X: \tilde U \to \Lin(V,V)$ which is $\FF^\times$-invariant and satisfies $X(\pp) = \tilde X(p)$. Hence, we can always think of vector fields as $\FF^\times$-invariants smooth functions defined on $\FF^\times$-stable open subsets of $V^\bullet$.
	
	If $t \in T_\pp\PP(V)$, with $\pp \in U$, then
	$$dX(t)=\frac{d}{d\varepsilon}\Big\vert_{\varepsilon =0} X\big(p+\varepsilon t(p)\big).$$
	Note that this derivative does not depend on the choice of a representative $p$ for $\pp$.
	
	\begin{defi} The connection $\nabla$ defined on the vector fields of $R(V)$ is given by 
		$$\nabla_t X (\pp) = \left(\frac{d}{d\varepsilon}\Big\vert_{\varepsilon =0} X\big(p+\varepsilon t(p)\big) \right)_\pp,$$ 
	where $X$ is a vector field, $\pp$ is a point on the domain of $X$, and $t \in T_\pp\PP(V)$. 
	\end{defi} 
	Here we are using the notation $M_\pp := \pi[\pp] \circ M \circ \pi'[\pp]$, where $M \in \Lin(V,V)$.
	So, the connection is defined as the derivative of vector fields up to the projections necessary to ensure that $\nabla_t X (\pp)$ is in the tangent space at $\pp$.
	If $X,Y$ are vector fields, then $\nabla_Y X$ is the vector field $\pp \mapsto \nabla_{Y(\pp)} X$. It is easy to see that $\nabla$ is a connection.
    
    The facts/expressions in \cite{coordinatefree} involving the connection hold in the case of $\FF$-modules as well:
    
    \begin{defi} \label{defi-adjoint}
	Given $t: V \to V$ and $\pp \in R(V)$ define $t^\ast: V \to V$ by the formula
	$$t^\ast v:= \frac{\langle v, tp \rangle}{\langle p , p \rangle}p.$$
	We call this function the {\bf adjoint} of $t$.
	\end{defi}
	
	\begin{lemma} Given $\pp \in R(V)$ and $t \in T_\pp \PP(V)$ we have
		$$\langle tu,v \rangle = \langle u,t^\ast v \rangle.$$
	\end{lemma}
	\begin{proof} We can write $t = t \circ \pi'[\pp]$ because $t \in T_\pp \PP(V)$. Just note that
	$$\langle tu,v \rangle = \frac{\langle u, p \rangle}{\langle p,p \rangle}\langle  tp,v \rangle,\quad
	\langle u,t^\ast v \rangle = \langle  u,p \rangle\frac{\langle v, tp \rangle^\ast}{\langle p,p \rangle}.$$
	So, $\langle tu,v \rangle = \langle u,t^\ast v \rangle.$
	\end{proof}
	\begin{lemma}[see Lemma 4.2 \cite{coordinatefree}] 
	\label{derivative of projection} Let $t \in T_\pp \PP(V)$ with $\pp \in R(V)$. Then
		$$\frac{d}{d\varepsilon}\Big\vert_{\varepsilon = 0} \pi'[p+\varepsilon tp] =- \frac{d}{d\varepsilon}\Big\vert_{\varepsilon = 0} \pi[p+\varepsilon tp] = t+t^\ast.$$
	\end{lemma}
	\begin{proof} The relation between the derivatives follows from $\pi[\xx]+\pi'[\xx] = \id$ for $\xx \in R(V)$. Now, note that
		$$\pi'[p+\varepsilon tp]v = \frac{\langle v, p+\varepsilon tp\rangle}{\langle p+\varepsilon tp,p+\varepsilon tp \rangle}(p+\varepsilon tp) = \frac{\langle v, p+\varepsilon tp\rangle}{\langle p,p \rangle+\varepsilon^2 \langle tp,tp \rangle}(p+\varepsilon tp).$$
		Since
		$$\frac{1}{\langle p,p \rangle+\varepsilon^2 \langle tp,tp \rangle} = \frac1{\langle p,p \rangle}+o(\varepsilon),$$
		$$\langle v, p+\varepsilon tp\rangle(p+\varepsilon tp) = \langle v,p \rangle p + \varepsilon\Big(\langle v,tp \rangle p + \langle v,p \rangle tp\Big)+ o(\varepsilon),$$
		and $$tv = \frac{\langle v,p \rangle }{\langle p,p \rangle}tp$$
		we conclude that
		$$\frac{d}{d\varepsilon}\Big\vert_{\varepsilon = 0} \pi'[p+\varepsilon tp] = t+t^\ast.$$
	\end{proof}

	The derivatives of a spread vector field with respect to a spread vector field is particularly simple:
	\begin{prop}[see Lemma 4.3 \cite{coordinatefree}] \label{covariant derivative of spread vector fields}Consider $t,s \in T_\pp \PP(V)$ with $\pp \in R(V)$. Let $T$ and $S$ be the vector fields spread from $t$ and $s$, respectively. Then
		$$\nabla_T S (\xx) = \left[ s\pi[\xx]t - t\pi'[\xx]s \right]_\xx.$$
	In particular, $\nabla_T S (\pp) =0$.
	\end{prop}
	\begin{proof} Since $S(x+\varepsilon T_{\xx}x) = \pi[x+\varepsilon T_{\xx}x] \circ s \circ \pi'[x+\varepsilon T_{\xx}x]$, we have, by Lemma \ref{derivative of projection},
		$$\frac{d}{d\varepsilon} \Big \vert_{\varepsilon=0} S(x+\varepsilon T_{\xx}x) = -(T_\xx+T_\xx^\ast) \circ s  \circ \pi'[\xx]+ \pi[\xx] \circ s \circ (T_\xx+T_\xx^\ast).$$
	From $\pi[x]\circ T_\xx^\ast =  T_\xx^\ast \circ \pi'[x] = 0$ we obtain
	$$\left[\frac{d}{d\varepsilon} \Big \vert_{\varepsilon=0} S(x+\varepsilon T_{\xx}x)\right]_\xx = -T_\xx \circ s  \circ \pi'[\xx]+ \pi[\xx] \circ s \circ T_\xx$$
	and $T_\xx = \pi[\xx] \circ t \circ \pi'[\xx]$ implies
	$$\nabla_T S(\xx) =  -\pi[\xx]\circ t \circ\pi'[\xx]\circ s  \circ \pi'[\xx]+ \pi[\xx] \circ s \circ \pi[\xx] \circ t \circ \pi'[\xx]= \left[ s\pi[\xx]t - t\pi'[\xx]s \right]_\xx.$$
	\end{proof}
	\begin{prop} \label{spread vector fields commutator vanishes} Consider $t,s \in T_\pp \PP(V)$ with $\pp \in R(V)$. Let $T$ and $S$ be the vector fields spread from $t$ and $s$ respectively. Then $[T,S](\pp)=0$.
	\end{prop}
	\begin{proof} Let $f \in C^\infty (\PP(V))$. We have
		$$S_\xx (f) = \frac{d}{d\varepsilon} \Big|_{\varepsilon=0} f\big(x + \varepsilon S_\xx(x)\big) = \frac{d}{d\varepsilon} \Big|_{\varepsilon=0} f\big(x + \varepsilon \pi[\xx]s(x)\big).$$
		So,
	 	$$T_\pp(Sf) = t(Sf) = \frac{d}{d\delta} \Big|_{\delta=0} \frac{d}{d\varepsilon} \Big|_{\varepsilon=0} f\big(p+\delta tp + \varepsilon \pi[p+\delta tp]s(p+\delta tp)\big).$$
	 It follows from Lemma \ref{derivative of projection} that
	 $$\pi[p+\delta tp]s(p+\delta tp) = \pi[\pp]sp + \delta( -(t+t^\ast)sp + \pi[\pp] stp ) + o(\delta) = sp - \delta t^\ast sp + o(\delta)$$
	 where we use that $t s = st =0$ and $\pi[\pp]sp = sp$ since $s,t \in T_\pp\PP(V)$. Hence,
	 $$T_\pp (Sf) = \frac{d}{d\delta} \Big|_{\delta=0} \frac{d}{d\varepsilon} \Big|_{\varepsilon=0} f(p+\delta tp+ \varepsilon sp),$$
	 which implies $T_\pp (Sf)=S_\pp (Tf)$.
	\end{proof}

	\begin{cor} \label{torsion free} The connection $\nabla$ is torsion free.
	\end{cor}
	\begin{proof} Follows immediately from Propositions \ref{covariant derivative of spread vector fields} and \ref{spread vector fields commutator vanishes}. 
	\end{proof}
	\begin{cor}[see Proposition 4.4 \cite{coordinatefree}]\label{compatibility with the kmetric} The connection $\nabla$ is compatible with the Hermitian metric.
	\end{cor}
	\begin{proof}Consider the tensor $B:=\nabla\langle -,- \rangle$, i.e., $B(S,T_1,T_2) := S \langle T_1,T_2 \rangle - \langle \nabla_S T_1, T_2 \rangle - \langle  T_1, \nabla_S T_2 \rangle$. Let $\pp \in R(V)$, let $t_1,t_2,s \in T_\pp \PP(V)$, and let $S,T_1,T_2$ be the vector fields spread respectively from $s,t_1,t_2$. By Proposition \ref{covariant derivative of spread vector fields} we have
		$$B(s,t_1,t_2) = s\langle T_1,T_2 \rangle.$$
	 Fix a representative $p$ for $\pp$ and define $u:=sp,\, v_1:=t_1p,\,v_2:=t_2p$. 
		 
		\begin{align*}s\langle T_1,T_2\rangle &= \left.\frac{d}{d\varepsilon}\right|_{\varepsilon =0} \langle {(T_1)}_{p+\varepsilon u}, {(T_2)}_{p+\varepsilon u} \rangle 
			\\&= \left.\frac{d}{d\varepsilon}\right|_{\varepsilon =0} \frac{\langle {(T_1)}_{p+\varepsilon u}(p+\varepsilon u),{(T_2)}_{p+\varepsilon u}(p+\varepsilon u) \rangle}{\langle p+ \varepsilon u,p+ \varepsilon u\rangle}.
		\end{align*}
		
		Since $\langle p , u \rangle =0$, we have $\langle p+ \varepsilon u,p+ \varepsilon u\rangle^{-1} = \langle p,p\rangle^{-1}+o(\varepsilon)$, and therefore
		$$s\langle T_1,T_2\rangle =\frac1{\langle p,p\rangle} \left.\frac{d}{d\varepsilon}\right|_{\varepsilon =0} \langle {(T_1)}_{p+\varepsilon u}(p+\varepsilon u),{(T_2)}_{p+\varepsilon u}(p+\varepsilon u) \rangle.$$
		
 		Now,
		\begin{align*}
			{(T_i)}_{\,p+\varepsilon u}(p+\varepsilon u)&=\pi[p+\varepsilon u]t_i\pi'[p+\varepsilon u](p+\varepsilon u)
			\\&=\pi[p+\varepsilon u]v_i
			\\&= v_i - \frac{\langle v_i,p+\varepsilon u \rangle}{\langle p+\varepsilon u,p+\varepsilon u\rangle}(p+\varepsilon u)
			\\&= v_i - \varepsilon  \frac{\langle v_i,u \rangle}{\langle p,p\rangle}p + o(\varepsilon),
		\end{align*}
		where we use that $tu=0$ and $\langle v_i,p \rangle =0$. Hence,
		$$\left.\frac{d}{d\varepsilon}\right|_{\varepsilon =0}{(T_i)}_{\,p+\varepsilon u}(p+\varepsilon u) =-\frac{\langle v_i,u \rangle}{\langle p,p\rangle}p $$
		and we obtain
		\begin{align*}
			s\langle T_1,T_2\rangle &=  - \frac{1}{\langle p ,p \rangle^2}\left(\langle v_1,u \rangle\langle p,v_2 \rangle+{\langle v_2,u \rangle^\ast}\langle v_1,p \rangle \right)=0,
		\end{align*}
		because $\langle v_i,p \rangle =0$. Thus, $B(s,t_1,t_2)=0$.
	\end{proof}

	Corollaries \ref{torsion free} and \ref{compatibility with the kmetric} say that the Hermitian and the pseudo-Riemannian metric are the Levi-Civita ones.
	
	\subsection{Geodesics}
	
	The geodesics in $\PP(V)$, as we will see in Proposition \ref{proposition: geodesics}, are of linear nature (analogous to the geodesics on a sphere).

	\begin{defi} \label{defi kgeodesic} Consider a $2$-dimensional real subspace $W$ of $V$ such that the restriction to $W$ of the Hermitian form $\langle \cdot,\cdot \rangle$ is $\RR$-valued and non-null. We call the projectivization $\PP_\FF(W)$ a {\bf geodesic}, where by $\PP_\FF(W)$ we mean the image of $W\cap V^\bullet$ under the quotient map $V^\bullet \to \PP_\FF(V)$.
	\end{defi}

		\begin{prop} The natural map $\phi:\PP_\RR (W\cap V^\bullet) \to \PP_\FF(V)$ is an immersion and its image is $\PP_\FF(W)$. Furthermore, if $\pp \in \PP_\FF (W)$ is regular, then 
		$$T_\pp \PP_\FF (W) = \{t \in T_\pp \PP_\FF (V): t(p) \in W\},$$
		where $p\in W$ is a representative of $\pp$.
	\end{prop}
	\begin{proof} Since the form restricted to $W$ is real and non-null, there exists $p\in W$ with $\langle p,p\rangle\ne0$. Consider an orthonormal basis $p,q$ for $W$. Observe that $\RR q$ contains the only points in $W$ which can be non-good since, for every $x\in W\setminus\RR q$ we have $\langle p,x\rangle\ne0$. So, 
	$W\cap V^\bullet$ is either $W\setminus\{0\}$ or $W \setminus \RR q$. Clearly, the image of $\phi$ is $\PP_\FF(W)$.
		
		We now prove that $\phi$ is injective. If $[\alpha p + \beta q] =[\alpha'p +\beta'q]$ in $\PP_\FF(W)$, where the coefficients $\alpha,\beta,\alpha',\beta'\in\RR$, then there exists $\gamma \in \FF^\times$ such that $\alpha p + \beta q =\gamma(\alpha'p +\beta'q)$. If $W\cap V^\bullet=W\setminus\{0\}$, then $p,q$ are $\FF$-linearly independent and it follows that $\alpha = \gamma \alpha'$ and $\beta = \gamma \beta'$.  If $W\cap V^\bullet=W \setminus \RR q$, then $\alpha \neq 0$ and, consequently, $\alpha \langle p,p \rangle = \gamma\alpha'\langle p,p \rangle$, implying that $\gamma \in \RR$. Thus  $[\alpha p + \beta q] =[\alpha'p +\beta'q]$ in $\PP_\RR(W \cap V^\bullet)$.
		
		The smoothness of $\phi$ follows from the commutative diagram bellow
		$$
		\begin{tikzcd}
			W \cap V^\bullet \arrow[d, "\mathrm{proj}_W"'] \arrow[r, "i", hook] & V^\bullet \arrow[d, "\mathrm{proj}_V"] \\
			\mathbb P_{\mathbb R}(W \cap V^\bullet) \arrow[r, "\phi"]                   & \mathbb P_{\mathbb F}(V)              
		\end{tikzcd}
		$$
		because the natural map $\proj_V \circ i:W \cap V^\bullet \to \PP_\FF(V)$ is smooth and $\RR^\times$-invariant. Let us show that $\phi$ is an immersion. If $v$ is a tangent vector at $\xx \in \PP_\RR (W \cap V^\bullet)$, then there is a vector $\tilde v \in W$ tangent to $W$ at $x$ such that $\de \proj_W(\tilde v)=v$. The image of $\tilde v$ by $\de i$ is $\tilde v \in V$ itself, and the image of this vector by $\de \proj_V$ is $\de \phi(v)$. If $\de \phi(v)=0$, then $\tilde v$ is tangent to the fiber of $\proj_V$ at $x$, which means $\tilde v = k x$ for some $k \in \FF$. It remains to show that $k\in\RR$ since this implies that $v = \de \proj_W(\tilde v)=0$. If $\langle x,x\rangle\ne0$, it follows from $\langle\tilde v,x\rangle=k\langle x,x\rangle$ that $k\in\RR$. Otherwise, assume $\langle x,x\rangle=0$ and $k\notin\RR$. Then $x,y:=\tilde v-x$ is a basis for $W$ satisfying $\langle x,x\rangle=\langle y,y\rangle=\langle x,y\rangle=0$, a contradiction. 
	\end{proof}
	
	In order to prove that (the regular parts of) the geodesics introduced in Definition \ref{defi kgeodesic} coincide with the geodesics of the Levi-Civita connection we will use a distinguished vector field introduced bellow.
	
	The {\bf tance } between two regular points $\pp,\qq$ is defined by
	$$\ta(\pp,\qq) := \frac{\langle p,q \rangle \langle q,p \rangle}{\langle p,p \rangle \langle q,q \rangle}.$$ 
	Clearly, the tance between two points is always a real number.
	
	Let $t \in T_\pp \PP(V)$, where $\pp \in R(V)$. We define the vector field $\tn(t)$  by the formula 
	$$\tn(t)(\xx):= \frac{T_\xx}{\ta(\pp,\xx)},$$
	where $T$ is the spread vector field from $t$. Note that $\tn(t)$ is a smooth vector field defined on the region described by $\ta(\pp,\xx) \neq 0$.
	The tance and the vector field $\tn(t)$ are extensions of the corresponding concepts introduced in \cite{discbundles} and \cite{coordinatefree}, respectively.
	
	Let us show that the integral curve of $\tn(t)$ starting at $\pp$ is the geodesic passing through $\pp$ with velocity $t$. We need the following lemma:
	
	\begin{lemma}[see Lemma 5.3 \cite{coordinatefree}] Let $\pp \in R(V)$ and $t \in T_\pp \PP(V)$. If $T$ is the spread vector field from $t$, then
		$$T_\xx \ta(\pp,\cdot) = -2 \ta(\pp,\xx) \re \frac{\langle tx,x \rangle}{\langle x,x \rangle}$$
	for all $\xx$ satisfying $\ta(\pp,\xx) \neq 0$.
	\end{lemma}
	\begin{proof} Let $\xi = T_\xx x$ for some representative $x$. By definition
		$$T_\xx \ta(\pp,\cdot) = \frac{\de}{\de \varepsilon}\Big|_{\varepsilon=0}\frac{\langle p, x+\varepsilon \xi \rangle\langle x+\varepsilon \xi,p \rangle}{\langle p, p \rangle\langle x+\varepsilon \xi,x+\varepsilon \xi \rangle}=2\re \frac{\langle p,x \rangle\langle \xi,p \rangle}{\langle p,p \rangle \langle x,x \rangle}.$$ 
	Since
	$$\xi = T_\xx x = \pi[\xx] tx = tx - \frac{\langle tx,x \rangle }{ \langle x,x \rangle}x$$
	and $tx \in p^\perp$ we obtain
	$$\langle \xi,p \rangle = -\langle tx,x \rangle \frac{\langle x,p \rangle}{\langle x,x \rangle},$$
	concluding the proof.
	\end{proof}
	\begin{prop}[see Thm 5.4 \cite{coordinatefree}] \label{proposition: geodesics} Let $t \in T_\pp \PP(V)$, $t\ne0$, with $\pp \in R(V)$. Consider the geodesic $\PP_\FF(W)$, where $W = \RR p+\RR tp$ {\rm(}note that $\PP_\FF(W)$ does not depend on the choice of representative $p$ for $\pp${\rm)}. 
	Let $c$ be a curve on $R(V)\cap \PP_\FF(W)$ satisfying
		$$c'(\theta) = \tn(t)_{c(\theta)}, \quad c(0) =\pp\quad \text{and}\quad c'(0) = t.$$
	The curve $c$ is the geodesic of the Levi-Civita connection passing through $\pp$ with velocity $t$.
	\end{prop}
	\begin{proof}
	Fix a representative $p$ of $\pp$ and a lift $\tilde c$ of $c$ such that $\tilde c(0)=p$. Since
	$$\frac{D}{d\theta } c'(\theta) = \nabla_{\tn(t)(c(\theta))}\tn(t)(c(\theta)),$$ it is enough to show that $\nabla_{T}\tn(t) = 0$.	By definition,
	\[\nabla_{T_\xx} \tn(t) = [\de \tn(t) (T_\xx)]_\xx .\]
	From $\tn(t) =\ta(\pp,\cdot)^{-1} T$ we obtain
	\begin{align*}
		\de\tn(t) (T_\xx) &=\frac{-1}{\ta(\pp,\xx)^2} (T_\xx\ta(\pp,\cdot))T_\xx + \frac1{\ta(\pp,\xx)} \de T (T_\xx)
		\\&=\frac{2}{\ta(\pp,\xx)} \re \frac{\langle tx,x \rangle}{\langle x,x \rangle } T_\xx + \frac1{\ta(\pp,\xx)} \de T (T_\xx).
	\end{align*}
	Taking into account that $\xx \in \PP_\FF(\RR p + \RR tp)$, we can take $x \in \RR p + \RR tp$. So,
	
	$$\nabla_{T_\xx} \tn(t)  =\frac{2}{\ta(\pp,\xx)} \frac{\langle tx,x \rangle}{\langle x,x \rangle } T_\xx + \frac1{\ta(\pp,\xx)} \nabla_{T_\xx} (T).$$
	By proposition \ref{covariant derivative of spread vector fields},
	$$\nabla_{T_\xx} T = [t \pi[\xx] t - t \circ \pi'[\xx] t]_\xx= \pi[\xx]t \pi[\xx] t\pi'[\xx] - \pi[\xx] t \pi'[\xx] t\pi'[\xx].$$
    Using that $t = \frac{\langle \,\cdot\, ,p \rangle}{\langle p,p \rangle}tp$,
	we have 
	\begin{align*}\pi[\xx]t \pi[\xx] t\pi'[\xx] &= \frac{\langle \,\cdot\, ,p \rangle}{\langle p,p \rangle}\pi[\xx]t\pi[\xx]tp=-\frac{\langle \,\cdot\, ,p \rangle}{\langle p,p \rangle}\frac{\langle tp ,x \rangle}{\langle x,x \rangle}\pi[\xx]tx,\\
	\pi[\xx] t \pi'[\xx] t\pi'[\xx] &=  \frac{\langle \,\cdot\, ,p \rangle}{\langle p,p \rangle}\pi[\xx]t\pi'[\xx]tp=\frac{\langle \,\cdot\, ,p \rangle}{\langle p,p \rangle}\frac{\langle tp ,x \rangle}{\langle x,x \rangle}\pi[\xx]tx,
	\end{align*}
	which implies
	$$\nabla_{T_\xx} T = -2 \frac{\langle \,\cdot\, ,p \rangle}{\langle p,p \rangle}\frac{\langle tp ,x \rangle}{\langle x,x \rangle}\pi[\xx]tx.$$
	On the other hand,
	$$T_\xx = \pi[\xx]\circ t \circ \pi'[\xx]= \frac{\langle \,\cdot\,,p \rangle}{\langle p,p \rangle}\pi[\xx]tp\quad\text{\rm and}\quad\langle tx,x \rangle = \frac{\langle x, p \rangle}{\langle p, p \rangle}\langle tp,x \rangle.$$
    So,
	$$\nabla_{T_\xx} \tn(t)  =\frac2{\ta(\pp,\xx)} \frac{\langle \,\cdot\,,p \rangle}{\langle p,p \rangle}\left( \frac{\langle tp,x \rangle}{\langle x,x \rangle} - \frac{\langle tp ,x \rangle}{\langle x,x \rangle} \right)\pi[\xx]tx=0.$$
	\end{proof}
	
	The following lemma allows us to explicitly find tangent vectors to geodesics.
	
		\begin{lemma}[see Lemma A.1 \cite{discbundles}] \label{velocity curve} Let $c$ be a smooth curve passing through $\pp\in R(V)$ at the instant $0$ and let $p$ be a representative of $\pp$. For any lift $\tilde c$ to $V^\bullet$ of the curve $c$ passing through $p$ at the instant $0$ we have
		$$c'(0) = \frac{\langle \,\cdot\,, \tilde c(0) \rangle}{\langle \tilde c(0),\tilde c(0) \rangle} \pi[\pp] \tilde c'(0).$$
	\end{lemma}
	\begin{proof} Consider a smooth funcion $f: \PP(V) \to \RR$. We have
		$$c'(0)f = \frac{d}{d \varepsilon}\Big\vert_{\varepsilon =0} f\big(\tilde c(\varepsilon)\big).$$
	Expanding $\tilde c$ in Taylor series we obtain $\tilde c(\varepsilon) = \tilde c(0) + \tilde c'(0) \varepsilon + o(\varepsilon)$ and, therefore, 
	$$c'(0)f = \frac{d}{d \varepsilon}\Big\vert_{\varepsilon =0} f\big(\tilde c(0) + \varepsilon\tilde c'(0)).$$
	Since the component of $\tilde c'(0)$ parallel to $p$ does not contribute to the derivative, we have
	$$c'(0)f = \frac{d}{d \varepsilon}\Big\vert_{\varepsilon =0} f\big(\tilde c(0) + \varepsilon \pi[\pp]\tilde c'(0)),$$
	implying the result.
\end{proof}
	
	Geodesics appear in three types. Consider $\pp \in R(V)$ and $t \in T_\pp\PP(V)$, $t\ne0$ with $\langle p,p\rangle\pm1$ and $\langle t,t\rangle\in\{-1,0,1\}$.
	
	Assume that the form on $W:=\RR p+\RR tp$ is nondegenerate definite. This means that $\langle tp,tp\rangle$ and $\langle p,p \rangle$ have the same sign and, therefore, $\langle tp,tp\rangle=\langle p,p\rangle$. We parametrize the geodesic $\PP_\FF(W)$ starting at $\pp$ with velocity $t$ by
	$$\theta \mapsto [\cos(\theta)p + \sin(\theta) tp].$$
	When the form on $W$ is nondegenerate indefinite, $\langle p,p \rangle$ and $\langle tp,tp \rangle$ have opposite signs and, therefore, $\langle p,p\rangle=-\langle tp,tp\rangle$. The geodesic $\PP_\FF(W)$ is now parametrized by
	$$\theta \mapsto [\cosh(\theta)p + \sinh(\theta) tp].$$
	Finally, when the form on $W$ is degenerate, that is, $\langle tp,W \rangle =0$, then the parametrization in question is
	$$\theta \mapsto [p + \theta tp].$$
	
	The verification that these curves are indeed geodesics follows from Lemma \ref{velocity curve} and Proposition~\ref{proposition: geodesics}.

		\begin{figure}[H]
	\centering
	\begin{minipage}{.5\textwidth}
		\centering
		\includegraphics[scale = .37]{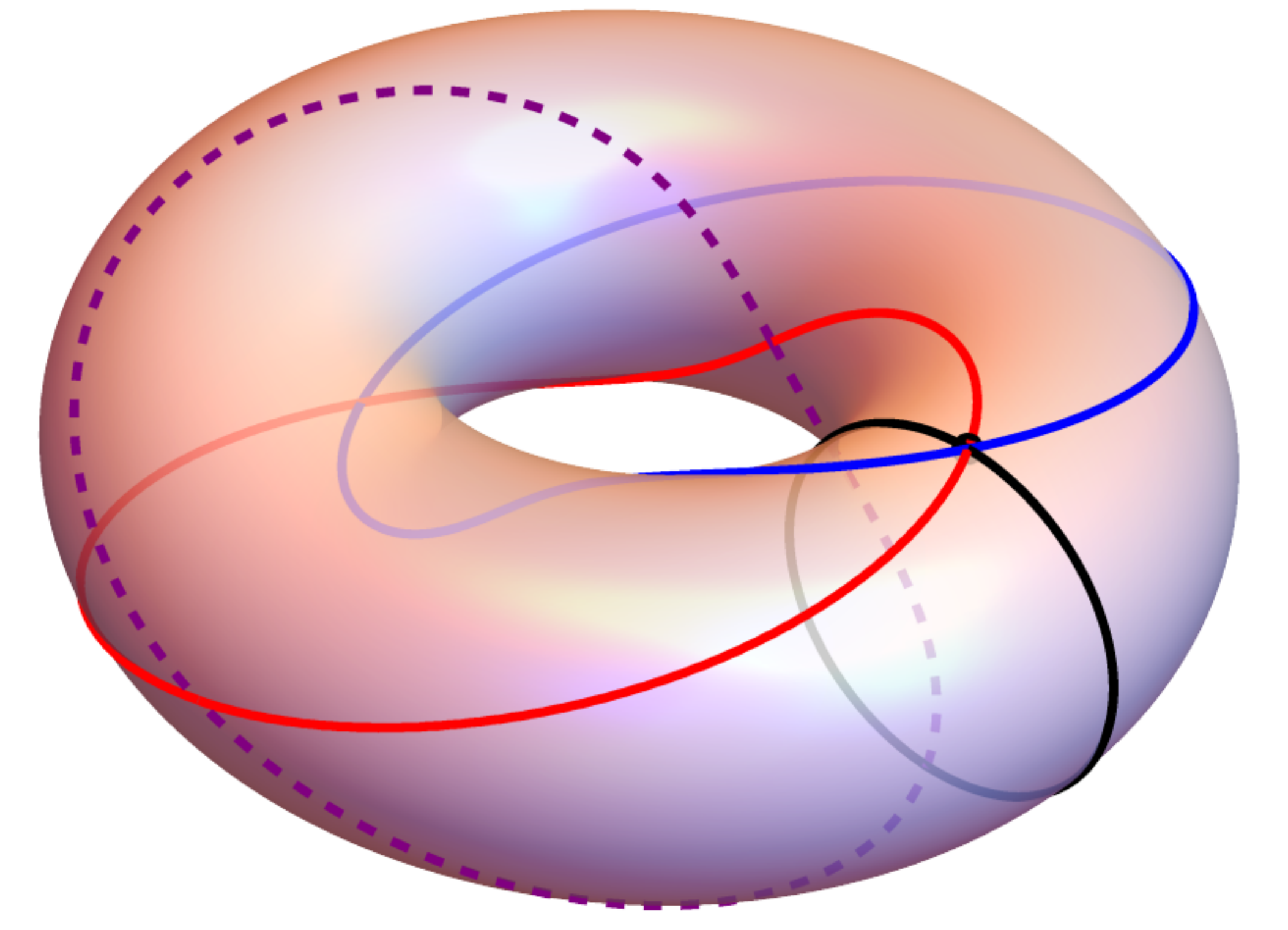}
		\caption*{(a)}
	\end{minipage}%
	\begin{minipage}{.6\textwidth}
		\centering
		\includegraphics[scale =.8]{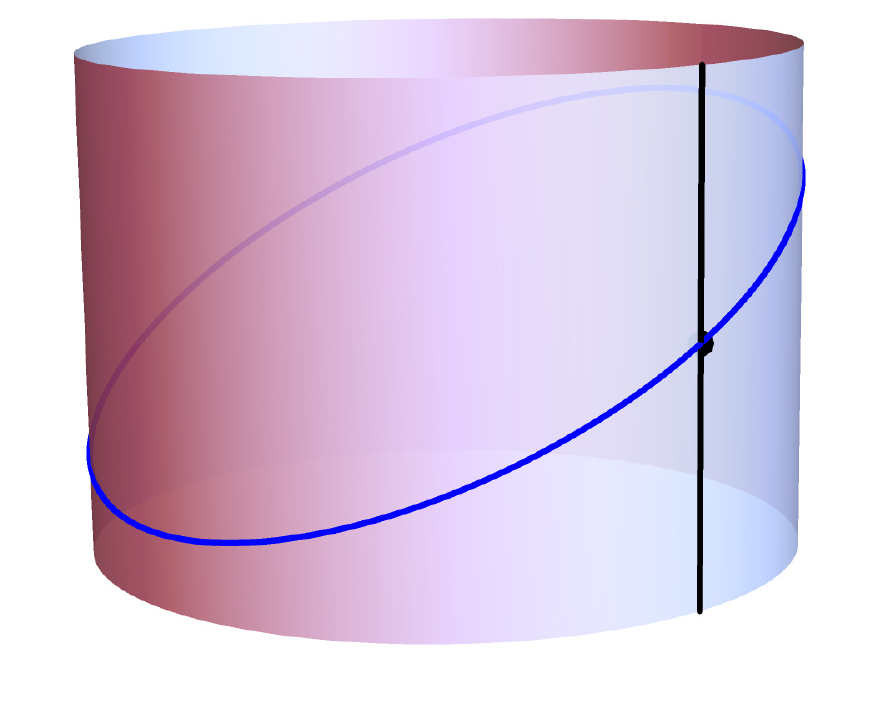}
		\vspace*{-0.2cm}
		\caption*{(b)}
	\end{minipage}
	\caption{ \textbf{(a)} Geodesics of the $\CC_s$-projective line and \textbf{(b)} geodesics of the $\DD$-projective line.}
	\label{Figure: projective lines for D and Cs}
\end{figure}

	\begin{exe} Consider the settings of the Example \ref{ex metrics} and fix a regular point $\pp$. Without loss of generality, we assume $\langle p,p \rangle =1$. Take a tangent vector $t$ at $p$ and consider a geodesic curve starting at $\pp$ with velocity $t$. We will say that this geodesic is positive, negative or null respectively when $\langle  t,t \rangle$ is positive, negative or null.
		
	Let us focus on geodesics in the projective lines. For the division algebras $\RR$, $\CC$ and $\HH$, the only geodesics appearing are the positive ones (the geodesics of the Fubini Study geometry). On the other hand,  for $\DD$ we have two types of geodesics: the positive and the null ones (tangent to the fibers). For the split-algebras $\CC_s$ and $\HH_s$, we have three types of geodesics: positive, negative and null. In Figure \ref{Figure: projective lines for D and Cs}, we represent the positive, negative and null geodesics by blue, red and black colors, respectively. 
	\end{exe}

	\begin{exe} \label{ex correspondence ds and hyperbolic} Consider $\RR^3$ with a Hermitian form of signature $-++$. As described in Example~\ref{example hyperbolic space}, the regular region of $\PP_\RR^2$ is formed by two connected components, the real hyperbolic plane $\HH_\RR^2$ and the projective de Sitter space $\mathrm{dS}^2$. The space of all non-oriented geodesics in the hyperbolic plane is the space $\mathrm{dS}^2$, a result obtained via point-plane duality. Indeed, for each $\pp \in \mathrm{dS}^2$, we have the geodesic $\PP(\pp^\perp)$, and all geodesics of $\HH_\RR^2$ are of this type (see Figure \ref{figure geodesics hyperbolic}). A positive geodesic of $\mathrm{dS}^2$ correspond to a one parameter family of geodesics in the hyperbolic plane sharing a common point in $\HH_\RR^2$; a null geodesic correspond to a family of geodesics meeting at a common point in the absolute $\partial \HH_\RR^2$; and a negative geodesic corresponds to a one parameter family of ultra-parallel geodesics on the hyperbolic plane perpendicular to it. 
	\end{exe}

	\begin{figure}[H]
		\centering
		\vspace*{-.4cm}
		\includegraphics[scale=.7]{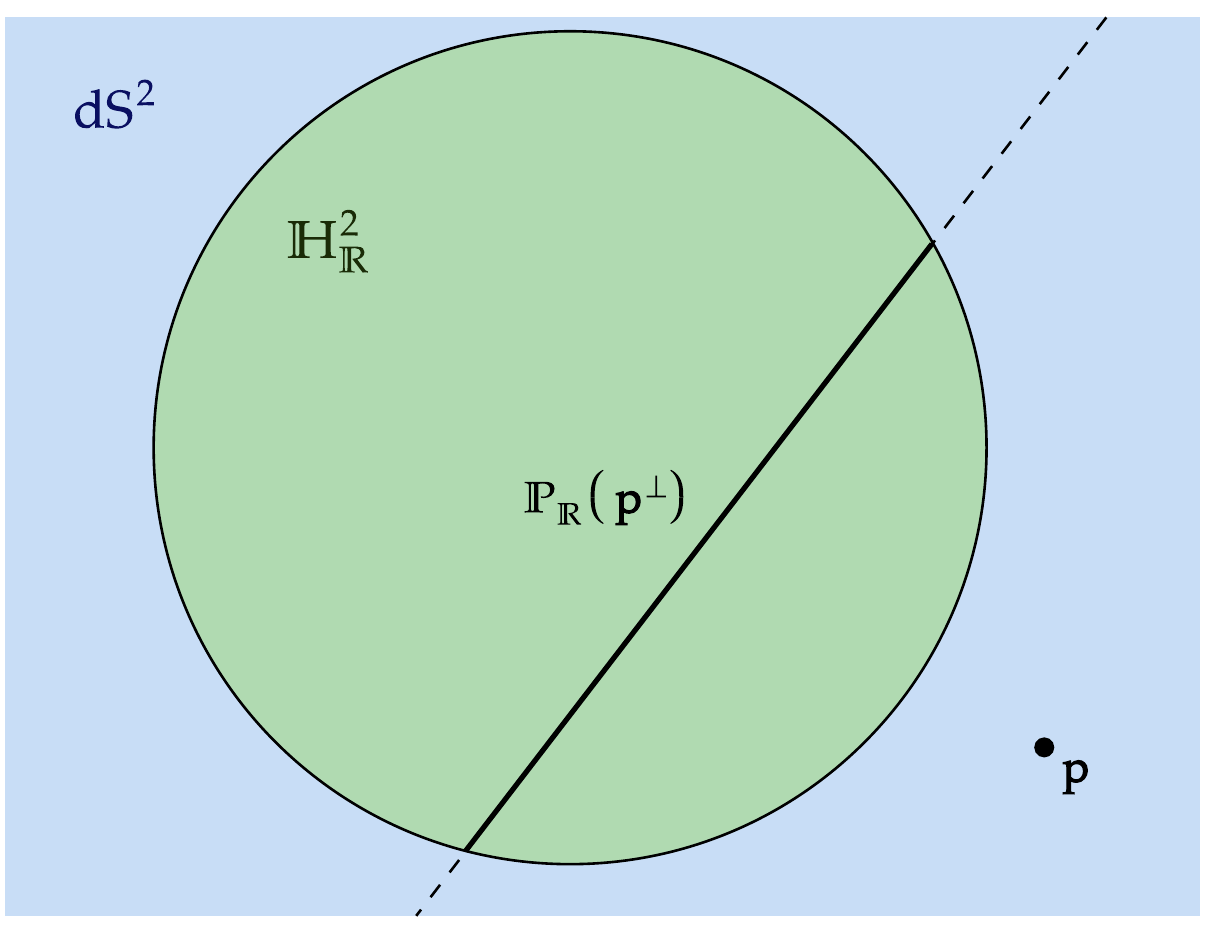}
		\caption{Real hyperbolic plane, projective de Sitter space, and space of geodesics.}
		\label{figure geodesics hyperbolic}
	\end{figure}
	
	\subsection{Curvature tensor}
	
	Let the Hermitian metric be given by
	$$\langle t_1,t_2 \rangle :=  \frac{\langle t_1(p), t_2(p) \rangle}{\langle p,p \rangle},$$
	where $t_1,t_2$ are tangent vectors at a regular point $\pp$. The pseudo-Riemannian metric $g$ is the real part of the Hermitian metric.
	
	The Riemann curvature tensor 
	$$R(U,V)W := \nabla_U \nabla_V W - \nabla_V \nabla_U W - \nabla_{[U,V]}W,$$
	defined for vector fields $U,V,W$, has a very simple expression in the settings of classical geometries~\cite{coordinatefree}:
	$$R(t_1,t_2)s = -s(t_1^* t_2 - t_2^* t_1)+(t_1 t_2^* - t_2 t_1^*)s,$$
	where $t_1,t_2,s \in T_\pp \PP_\FF(V)$, where $s^*$ is the adjoint of $s$ (see Definition \ref{defi-adjoint}) and the same goes for $t_1^*$, $t_2^*$. The proof of this fact follows from using spread vector fields (Definition \ref{spread vector fields}) and  applying Propositions \ref{covariant derivative of spread vector fields} and \ref{spread vector fields commutator vanishes}.
	
	For a real subspace $W = \RR t_1 \oplus \RR t_2$ of $T_\pp \PP_\FF(V)$, where $t_1,t_2$ are tangent vectors at $\pp$ orthonormal with respect to $g$, the sectional curvature
	$$K(W) := \frac{g(R(t_1,t_2)t_2,t_1)}{g(t_1,t_1) g(t_2,t_2) - g(t_1,t_2)^2}$$
	is given by
	$$K(W) = \frac{\langle t_1,t_1 \rangle  \langle t_2,t_2\rangle - 2 \re \langle t_1,t_2 \rangle^2  + \langle t_1,t_2\rangle \langle t_2,t_1 \rangle}{\langle t_1,t_1 \rangle  \langle t_2,t_2\rangle}.$$
	Writting $a_i := \langle t_i,t_i\rangle$ and $b:=\langle t_1,t_2\rangle$, we have $a_i \in \{1,-1\}$ and $\re b =0$, because $t_1,t_2$ is an orthonormal basis for $W$ with respect to $g$. Thus,
	$$K(W) = 1+ \frac{b b^* - 2\re b^2}{a_1a_2}=1 - \frac{3 b^2}{a_1a_2}$$
	where the last equality holds for the algebras under consideration.

	When $\FF= \RR$, the curvature is constant and equals $1$. For the dual numbers, the same happens because $b^2=0$. Note that the tangent planes to points in the dual numbers projective line do not possess a non-degenerate real two-dimensional subspace plane $W$. Thus, sectional curvatures is not defined in this case (that is not what happens in higher dimensions). The cases where $\FF = \CC,\HH$ are detailed at \cite[Section 4.6]{coordinatefree}.
	
	Now we analyse the case where $\FF$ is $\CC_s$ or $\HH_s$. If $V$ is two-dimensional, then $T_\pp \PP_\FF(V)$ has dimension one as an $\FF$-module. Therefore, $t_1 = k t_2$ for some $k \in \FF^\times$. Note that $b = k a_2$ and $a_1 = - k^2a_2$. Therefore,
	$$K(W) = 1 - \frac{3 b^2}{a_1a_2} = 4.$$
	
	If $V$ has $\FF$-dimension higher than $2$, then $K(W)$ can be any real number. Indeed, consider the tangent vectors $e_1$, $e_2$ at $\pp$ such that $\langle e_1,e_1 \rangle = \langle e_2,e_2 \rangle=1$ and $\langle e_1,e_2 \rangle =0$, which exist because $T_\pp \PP_\FF(V)$ is at least two dimensional as an $\FF$-module. For $t_1=e_1$ and $t_2 =\sinh(\theta)je_1+\cosh(\theta)e_2$, we obtain $K(W) = 1-3\sinh(\theta)^2$. For $t_1=e_1$ and $t_2 =\cosh(\theta)je_1+\sinh(\theta)e_2$, we obtain
	$K(W) = 1+3\cosh(\theta)^2$. Finally, for $t_1 = e_1$ and $t_2 = j(\cos(\theta)e_1 + \sin(\theta)e_2)$, we have $K(W) = 1+3\cos(\theta)^2$.
	
	\smallskip
	
	Summarizing, for the algebras other than the division ones, we have: for dual numbers, the curvature is always one; for split algebras, the curvature equals $4$ in the projective line case and can be any number otherwise.
	
	Note that, had we taken the Hermitian metric with a negative sign, then the curvature formula would have its sign changed as well. For the hyperbolic models, for instance, we take the negative sign (Example \ref{example hyperbolic space}). Thus, the real hyperbolic spaces have curvature $-1$. For the complex and quaternionic hyperbolic spaces the obtained curvature is $-4$ for one dimensional spaces and, for higher dimension, the curvature lies in $[-4,-1]$. 
	
	\section{Spaces of oriented geodesics on Euclidean, elliptical and hyperbolic two-dimensional geometries.}
	\label{section space of geodesics}
	
	In the following examples, we consider the algebras $\FF = \CC$, $\DD$, $\CC_s$ and the $\FF$-module $\FF^2$ endowed with the Hermitian form $\langle u,v \rangle = u_1 v_1^\ast + u_2 v_2^\ast$. We will see that the regular components of the spaces $\PP_\CC^1$, $\PP_\DD^1$ and $\PP_{\CC_s}^1$ are the spaces of geodesics of the round sphere, Euclidean plane and hyperbolic plane, respectively.
		
	\subsection{Points in the complex projective line = oriented geodesics in $\SP^2$}
	In these settings, the Riemann sphere $\PP_\CC^1$ is a constant curvature sphere. So, a given point $\pp\in\PP_\CC^1$ determines a unique equator (the geodesic equidistant from $\pp$ and its antipodal point). This geodesic is oriented in the counterclockwise direction as seen from $\pp$. The Hermitian metric measures the oriented angle between two oriented geodesics and at an intersection point.
	
	\subsection{Points in the dual number projective line = oriented geodesics in $\mathbb E^2$}
	Identify $\EE^2$ with the complex plane $\CC$. The cylinder $\SP^1 \times \RR$ can be identified with $T \SP^1$ via the map $(e,s) \mapsto (e,sie)$, where $\SP^1$ is taken as the circle of unit complex numbers. An oriented line in the Euclidean plane $\EE^2$ is given by 
	$c_{e,s}(t):=s ie+ et$, where $e$ is a unit vector and $s$ is a real number. Thus, we have a one-to-one correspondence between points in the tangent bundle $T \SP^1$ and the space of oriented lines in the plane which is given by $\SP^1 \times \RR \ni (e,s) \mapsto c_{e,s}$.
	
	Fix a point $p= a+ib$ in the plane. The lines passing through $p$ are $c(e(\theta), s(\theta))$, where $e(\theta) = \exp(i\theta)$ and $s(\theta) := -a \sin(\theta)+b\cos(\theta)$.
	Thus, taking the coordinates $(x+iy,s) \in \CC \times \RR$, where $\CC \times \RR$ is the ambient space of the cylinder $\SP^1 \times \RR$, the previously described family of oriented lines is obtained by intersecting the linear subspace $s = b x-a y$ with the cylinder. Thus, families of oriented lines sharing a fixed point $p$ correspond to planes that cut the cylinder in an ellipse.
	Each of the remaining planes cut the cylinder in two components (lines); one of them corresponds to a family of oriented geodesics in $\EE^2$ and, the other, to the same family of geodesics in $\EE^2$ with the opposite orientation.
	
	The described curves in $T \SP^1$ are geodesics of the following metric on the cylinder:
	$$g_{T\SP^1}((u_1,s_1),(u_2,s_2)) := \re u_1\overline{u_2}.$$ 
	Thus, the distance between two points on $T\SP^1$ is the angle between the corresponding oriented lines.
	
	Now, we just have to identify the described cylinder with the dual number projective line. Take $V=\CC+\ee\, \CC$ as a $\DD$-module. The diffeomorphism $f:\PP_\DD^1 \to \SP^1 \times \RR$ given by $[e+ k \ee i e] \mapsto (e^2, 2k)$ is the desired isometry (up to rescaling the metrics), where $e$ is a unit complex number. Indeed, $df^\ast g_{T\SP^1} = 4 g_{\PP_\DD^1}$. Therefore, $\PP_\DD^1$ is the space of all oriented Euclidean lines. On the dual numbers projective line, a positive geodesic represents a family of oriented lines rotating around a common point in the Euclidean plane while a null geodesic represents a one parameter family of parallels lines (see Figure \ref{Figure: projective lines for D and Cs}(b)).  
	
	\subsection{Regular points of the split-complex projective line = oriented geodesics in $\HH_\RR^2$} As we discussed in Example \ref{ex correspondence ds and hyperbolic}, the projective de Sitter space $\mathrm{dS}^2$ is the space of all non-oriented hyperbolic geodesics.  Topologically, $\mathrm{dS}^2$ is an open Möbius strip. The regular part of the split-complex projective line is a cylinder (see Figure \ref{Figure: projective lines for D and Cs}(a); the regular region is the cylinder obtained from removing the singular circle, the dashed curve in purple, from the torus) and it is an isometric double cover of $\mathrm{dS}^2$ (up to rescaling the metrics). Furthermore, it constitutes the space of oriented geodesics of the hyperbolic plane.

	The cross product on $\RR^3$ endowed with the canonical Minkowski metric $-++$ is given by $e_1 \times e_2 = e_3$, $e_2 \times e_3 = - e_1$, $e_3 \times e_1 = e_2$, where $e_1,e_2,e_3$ is the canonical basis of $\RR^3$. Equivalently, the cross product is defined by the formula $\langle u \times v, w \rangle e_1 \wedge e_2 \wedge e_3 = u \wedge v \wedge w$.
	Given two points $\pp,\qq \in \partial \HH_\RR^2$, the vector $p \times q$ represents the point in $\mathrm{dS}^2$ corresponding to the geodesic $G$ of the hyperbolic plane connecting $\pp$ and $\qq$, i.e., $G = \PP_\RR\big((p\times q)^\perp\big)$.
	
	Taker $V=\CC_s^2$ with the Hermitian form defined in Example \ref{ex metrics}.
	For each $[(a,a'):(b,b')] \in R(V)$ we have the points $A(a,b):=[a^2+b^2:a^2-b^2:2ab]$ and $B(a',b'):=[a'^2+b'^2:-a'^2+b'^2:-2a'b']$ on $\partial \HH_\RR^2$ and, thus, the oriented geodesic of the hyperbolic plane connecting $B(a',b')$ to $A(a,b)$. Observe that the condition for $[(a,a'):(b,b')]$ to be in $R(V)$ is $aa'+bb'\neq 0$, and the same condition guarantees $A(a,b) \neq B(a',b')$ in $\PP_\RR^2$. Therefore, we obtain a correspondence between $R(V)$ and oriented geodesics in $\HH_\RR^2$. The point $A(a,b) \times B(a',b')$ in $\mathrm{dS}^2$ corresponds to the non-oriented geodesic containing $A(a,b)$ and $B(a',b')$ (see Example \ref{ex correspondence ds and hyperbolic}).
	\begin{figure}[H]
	\centering
	\vspace*{-.4cm}
	\includegraphics[scale=.7]{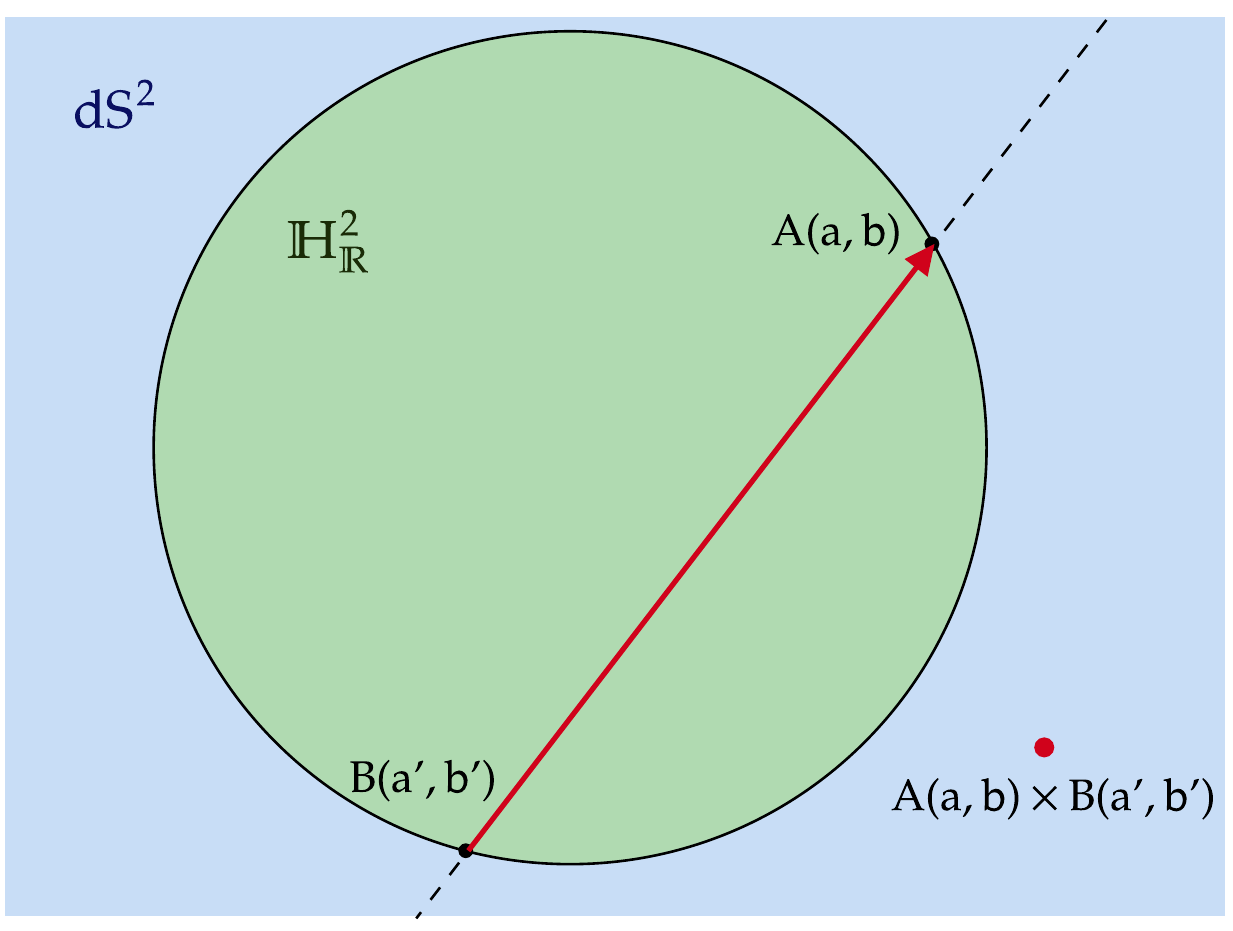}
	\caption{Points $A(a,b)$, $B(a',b')$ and $A(a,b) \times B(a',b')$.}
	\label{figure space of lines}
	\end{figure}
	
	The double cover $f:R(V) \to \mathrm{dS}^2$ of interest is given by
	$[(a,a'):(b,b')]\mapsto A(a,b)\times B(a',b')$,
    that is,
	$$f([(a,a'):(b,b')]) = [ab'-a'b:ab'+a'b:-aa' + bb'].$$
	A direct computation, sketched bellow, shows that $df^\ast g_{\mathrm{dS}^2} = -4 g_{\PP_{\CC_s}^1}$. Thus, up to rescaling the metrics, $f$ is an isometric $2$ to $1$ cover map.
	
	In order to verify that $df^\ast g_{\mathrm{dS}^2} = -4 g_{\PP_{\CC_s}^1}$, consider a point $\pp \in R(V)$ with representative $p=((a,a'),(b,b'))$. We assume that $aa'+bb' =1$. The tangent vectors 
	$t_1 = \langle \cdot, p \rangle v$ and ${t_2 = \langle \cdot, p \rangle (1,-1)v}$ at $\pp$, where $v=((b',b),(-a',-a))$,
	satisfy $g_{\PP_{\CC_s}^1}(t_1,t_1) = 1$, $g_{\PP_{\CC_s}^1}(t_2,t_2)=-1$, and $g_{\PP_{\CC_s}^1}(t_1,t_2)=0$. 
	The curves $\gamma_1(\theta) = [p+\theta v]$ and $\gamma_2(\theta) = [p+\theta (1,-1)v]$ have respective velocities $t_1$ and $t_2$ at $\theta =0$. 
	
	Now we compute $df(t_i)$ using the curve $\gamma_i$. We have
	$f \circ \gamma_1(\theta) = [q+\theta w_1+o(\theta)]$, where
	$$q =(ab' -a'b,ab'+a'b,-aa'+bb'), \quad \text{and}\quad w_1=(-a^2-b^2+a'^2+b'^2,-a^2+b^2-a'^2+b'^2),$$ and, by Lemma \ref{velocity curve}, we obtain $df(t_1) = \langle \cdot,q \rangle w_1$.
	Similarly, $f \circ \gamma_2(\theta) = [q+\theta w_2+o(\theta)]$ and $df(t_2)=\langle \cdot, q \rangle w_2$, where $$w_2 = (a^2+b^2+a'^2+b'^2,a^2-b^2-a'^2+b'^2,2ab-2a'b').$$
	The identity $df^\ast g_{\mathrm{dS}^2} = -4 g_{\PP_{\CC_s}^1}$ follows from the fact that $$df^\ast g_{\mathrm{dS}^2}^2(t_1,t_1)=-4, \quad df^\ast g_{\mathrm{dS}^2}(t_2,t_2)=4, \quad \text{and} \quad df^\ast g_{\mathrm{dS}^2}(t_1,t_2)=0.$$
	
	\medskip
	
	Considering the split-quaternions $\HH_s$ and the module $\HH_s \times \HH_s$, we obtain an ambient space for a transition between the three described geometries. Indeed, with the Hermitian form
	$\langle u,v \rangle =  u_1 v_1^\ast + u_2 v_2^\ast$ on $\HH_s \times \HH_s$, the maps in the Example \ref{example split-quaternionic projective spaces} are isometric embeddings (when restricted to the regular region). Thus, there exists a natural transition between the regular regions of the three discussed $2$-dimensional geometries inside the split-quaternionic projective line.
	
	\begin{rmk} Taking $\HH_s\times \HH_s$ with the Hermitian form $\langle u,v \rangle = - u_1 v_1^\ast + u_2 v_2^\ast$, we obtain the hyperbolic spaces $\HH_\CC^1$, $\HH_\DD^1$ and $\HH_{\CC_s}^1$, where $\HH_\FF^1$ is formed by the regular points $\pp$ admiting a representative $p$ satisfying $\langle p,p \rangle < 0$. As above, we can geometrically transition between this geometries inside the split-quaternionic projective line. A transition between hyperbolic geometries is also studied in \cite{tre} via a more abstract route. In contrast, here we use that these geometries share a common ambient space.
	\end{rmk}
	\section{Bidisc geometry}
	\label{section bidisc}
	The bidisc is the Riemannian manifold $\HH_\CC^1 \times \HH_\CC^1$, the product of two Poincaré discs, with the canonical Riemannian product metric. The metric we take in the Poincaré disc $\HH_\CC^1$ is the one defined in Example \ref{example hyperbolic space}. In this section, we want to show how the bidisc appears as part of a projective line. For that purpose, we use an algebra not previously considered.
	
	Let $\FF$ be the real algebra $\CC \times \CC$ with the involution $(a,b)^* = (\overline a, \overline b)$.  The algebra of self-adjoint elements of $\FF$ in this case is $\RR \times \RR$. 
	The projective line $\PP_\FF^1$ is diffeomorphic to the product of two Riemann spheres $\PP_\CC^1 \times \PP_\CC^1$. Indeed, the diffeomorphism is given by the map $\Lambda:\PP_\FF^1 \to \PP_\CC^1 \times \PP_\CC^1$, $[(a_1,b_1),(a_2,b_2)] \mapsto ([a_1,a_2],[b_1,b_2])$.
	
	Consider in $\FF^2$ the $\FF$-valued Hermitian form
	$$\langle u,v \rangle = -u_1v_1^* + u_2v_2^*.$$
	For $u=((a_1,b_1),(a_2,b_2))$,
	$$\langle u,u \rangle = (- |a_1|^2 +|a_2|^2,-|b_1|^2 +|b_2|^2) \in \RR \times \RR.$$
	We define the regular region $R$ of $\PP_\FF^1$ as the set of all $\uu$ such that $\langle u,u \rangle$ is a unit, which means that both coordinates of $\langle u,u \rangle$ are non-zero real numbers.
	Observe that $R$ is the union of four disjoint $4$-balls. Indeed, these balls are
	$$B_{++} = \{\uu\in\PP_\FF^1\mid\langle u,u \rangle \in \RR_{>0} \times \RR_{>0}\},$$
	$$B_{- +} = \{\uu\in\PP_\FF^1\mid\langle u,u \rangle \in \RR_{<0} \times \RR_{>0}\},$$
	$$B_{+ -} = \{\uu\in\PP_\FF^1\mid\langle u,u \rangle \in \RR_{>0} \times \RR_{<0}\},$$
	$$B_{- -} = \{\uu\in\PP_\FF^1\mid\langle u,u \rangle \in \RR_{<0} \times \RR_{<0}\},$$
	and $R = B_{++} \sqcup B_{+-} \sqcup B_{- +} \sqcup B_{- -}$.
	We denote $B_{--}$ by $\BB^4$.

	Let $\lambda:\FF^2 \to \CC^2 \times \CC^2$ be given by the formula $((a_1,b_1),(a_2,b_2)) \mapsto ((a_1,a_2), (b_1,b_2))$. If $\pi_1, \pi_2$ stand respectively for the projections $\CC^2 \times \CC^2 \to \CC^2$ in the first and second coordinates, we define $\lambda_1  = \pi_1 \circ \lambda$ and $\lambda_2  = \pi_2 \circ \lambda$.

	Consider on $\CC^2$ the $\CC$-valued Hermitian form $$\langle (a_1,a_2),(a_1',a_2') \rangle := - a_1\overline{a_1'}+a_2\overline{a_2'}.$$ 
	For $u,u' \in \FF^2$ we have
	$$\langle u, u' \rangle = (\langle \lambda_1(u),\lambda_1(u') \rangle, \langle \lambda_2(u),\lambda_2(u') \rangle).$$

	The complex hyperbolic line $\HH_\CC^1$ is formed by $\pp \in \PP_\CC^1$ such that $\langle p, p \rangle <0$. Therefore, the map $\Lambda: \PP_\FF^1 \to \PP_\CC^1 \times \PP_\CC^1$ provides a diffeomorphism between $\BB^4$ and $\HH_\CC^1 \times \HH_\CC^1$. The ball $\BB^4$ will be our projective model for the bidisc.
	
    A unitary operator $T:\FF^2 \to \FF^2$ is a $\FF$-linear map satisfying $\langle T(u) , T(v) \rangle = \langle u,v  \rangle$. Writting $T$ as the matrix
	$$\begin{pmatrix}
    (a_{11},b_{11}) & (a_{12},b_{12})\\
    (a_{21},b_{21}) & (a_{22},b_{22})
    \end{pmatrix}$$ 
    we obtain that $T$ is unitary if, and only if, the matrices $(a_{ij})$ and $(b_{ij})$ are unitary as well.
    Thus we have the map $\UU(\FF^2, \langle \cdot, \cdot \rangle) \to \UU(1,1) \times \UU(1,1)$, $(a_{ij},b_{ij})\mapsto \big((a_{ij}),(b_{ij})\big)$ which is a group isomorphism.
    The action of unitary transformations on $\BB^4$ correspond to the action of ${\UU(1,1) \times \UU(1,1)}$ on $\HH_\CC^1 \times \HH_\CC^1$. If we restrict ourselves to determinant $1$ matrices, the above isomorphism holds for $\SU$ matrices as well: $\SU(\FF^2, \langle \cdot, \cdot \rangle) \simeq \SU(1,1) \times \SU(1,1)$. The same goes for projective unitary group: $\PU(\FF^2, \langle \cdot, \cdot \rangle) \simeq \PU(1,1) \times \PU(1,1)$.

    Now we consider the Hermitian metric on $\HH_\CC^1$ introduced in Example \ref{example hyperbolic space}. In a similar fashion, we consider the $\FF$-valued Hermitian metric
    $$\langle t,t' \rangle = -\frac{\langle t(p),t'(p) \rangle}{\langle p,p \rangle}$$
    for $t,t' \in T_\pp \PP_\FF^1$, where $\pp$ is a regular point.
    Writing the tangent vector $t\in T_\pp \PP_\FF^1$ as $$t = \frac{\langle \cdot , p \rangle}{\langle p ,p \rangle }t(p)$$
    we obtain that its image in $\PP_\CC^1 \times \PP_\CC^1$ is $(s_1,s_2)$, where
    $$s_j:=\frac{\langle \cdot, \lambda_j(p) \rangle}{\langle \lambda_j(p), \lambda_j(p) \rangle} \lambda_j(t(p)).$$
    Therefore, the Hermitian metric on the $\FF$ projective line corresponds to the pair of Hermitian metrics arising from the two Riemann spheres. More precisely
    $$\langle t,t' \rangle = (\langle s_1,s_2 \rangle, \langle s_1',s_2' \rangle).$$
    
    From this $\FF$-Hermitian metric, we obtain a Riemannian metric by taking the real part of the $\FF$-value Hermitian metric, which gives an element of $\RR \times \RR$, and then summing the obtained coordinates. Let us denote this metric by $g_{\BB^4}$. Therefore, the map $\Lambda: \BB^4 \to \HH_\CC^1 \times \HH_\CC^1$ is an isometry:
    $$g_{\BB^4}(t,t') = g_{\HH_\CC^1}(s_1,s_2) + g_{\HH_\CC^1}(s_1',s_2').$$
    
    Finally, the group of orientation preserving isometries of the bidisc $\HH_\CC^1 \times \HH_\CC^1$ is generated by $\PU(1,1) \times \PU(1,1)$ and the map that swaps the coordinates of the two hyperbolic discs ${\tau:(\pp,\qq) \mapsto (\qq,\pp)}$. In~$\BB^4$, this map $\tau$ is given by $\tau:[(a_1,b_1):(a_2,b_2)] \mapsto [(b_1,a_1):(b_2,a_2)].$
    Hence, $\BB^4$ is a projective model for the bidisc and the unitary group of $(\FF^2, \langle \cdot , \cdot \rangle)$ together with $\tau$ provides the orientation preserving isometries.
    Furthermore, there exists in $\BB^4$ an orientation preserving isometry sending the pair of points $\uu,\uu'$ to the pair of points $\vv,\vv'$ if, and only if, either $\ta(\uu,\uu') = \ta(\vv,\vv')$ or $\ta(\uu,\uu') = \ta(\tau \vv,\tau \vv')$, where the tance here is $\RR \times \RR$-valued, obtained from the $\FF$-valued Hermitian form defined on $\FF^2$.


\newpage
	\begin{thebibliography}{90} 
	
    \bibitem[AGG]{discbundles} Sasha Anan'in, Carlos H. Grossi, Nikolay Gusevskii, {\it Complex Hyperbolic Structures on Disc Bundles over Surfaces.} Int.~Math.~Res.~Not., Vol. 2011 {\bf19}, p.~4295--4375.

    
    \bibitem[AGr]{coordinatefree} Sasha Anan'in, Carlos H. Grossi, {\it Coordinate-Free Classic Geometries.} Mosc.~Math.~J., Vol.~11, Number 4, October-December 2011, p.~633--655.
    
    \bibitem[BGr]{bgr} Hugo C. Botós, Carlos H. Grossi. {\it Quotients of the holomorphic 2-ball and the turnover.}	arXiv:2109.08753
    
    \bibitem[CGr]{cgr1} Sidnei F. Costa, Carlos H. Grossi. {\it Disc bundles over surfaces uniformized by the holomorphic bidisc.}  In preparation.
    
    
    \bibitem[Dan1]{dan1} J. Danciger. {\it Geometric transitions: From hyperbolic to ads geometry.} PhD Thesis, Stanford University, 2011.
    
    \bibitem[Dan2]{dan2} J. Danciger. {\it A geometric transition from hyperbolic to anti-de Sitter geometry.} Geom. Topol., Vol.~17, 2013, p.~3077–3134.
    
    \bibitem[Dan3]{dan3} J. Danciger. {\it Ideal triangulations and geometric transitions.} Journal of Topology. J.~Topol., Vol.~7, 2014.
    
    \bibitem[GKL]{gkl} Goldman, W. M., M. Kapovich, and B. Leeb. {\it Complex hyperbolic manifolds homotopy equivalent to a Riemann surface.} Communications in Analysis and Geometry 9, no. 1 (2001): 61--95
    
    \bibitem[GLT]{glt} M. Gromov, H. B. Lawson Jr., W. Thurston, {\it Hyperbolic 4-manifolds and conformally flat 3-manifolds,} Inst.~Hautes Études Sci.~Publ. Math. {\bf68} (1988), 27--45
    
    \bibitem[Sh]{sh} Every real variety contains non-singular points - Adam Sheffer. \\Link: \url{https://mathoverflow.net/q/262589}.
    
    \bibitem[Tre]{tre} Steve J. Trettel, {\it Families of Geometries, Real Algebras, and Transitions.} UC Santa Barbara, Ph.D. dissertation (2019).
	\end{thebibliography}
\end{document}